\documentclass[12pt, reqno]{amsart}

\usepackage[paperheight=11in,paperwidth=8.5in,left=1in,top=1.25in,right=1in,bottom=1.25in,]{geometry}
\usepackage[linktocpage=false,colorlinks=true,linkcolor=black,citecolor=black,
filecolor=black,urlcolor=black,pagebackref=false,pdfstartpage={1},
pdfstartview={FitH},pdftitle={},pdfauthor={},pdfsubject={},pdfcreator={},
pdfproducer={},pdfkeywords={}]{hyperref}

\usepackage[htt]{hyphenat}

\usepackage{rotating}
\usepackage{pdflscape}
\usepackage{manfnt}

\usepackage{chngpage}
\usepackage{diagbox}
\usepackage[lofdepth,lotdepth]{subfig}
\usepackage{comment}
\usepackage{afterpage}
\usepackage{amsfonts}
\usepackage{amsmath}
\allowdisplaybreaks[4]
\usepackage{amssymb}
\usepackage{amsthm}
\usepackage{bbm}
\usepackage[justification=centering,margin=1cm]{caption}
\usepackage[capitalize,nameinlink,noabbrev,nosort]{cleveref}
\usepackage{enumerate}
\usepackage{enumitem}
\usepackage{float}
\restylefloat{figure}
\usepackage{youngtab}
\usepackage{psfrag}
\usepackage[aligntableaux=center]{ytableau}

\usepackage{graphicx}
\usepackage{mathrsfs}
\usepackage{mathtools}
\usepackage[framemethod=tikz,ntheorem,xcolor]{mdframed}
\usepackage{parcolumns}
\usepackage{tabularx}
\usepackage{tikz}\pdfpageattr{/Group <</S /Transparency /I true /CS /DeviceRGB>>}
\usetikzlibrary{arrows,calc,intersections,matrix,positioning,through}
\usepackage{tikz-cd}
\tikzset{commutative diagrams/.cd,every label/.append style = {font = \normalsize}}
\usepackage[all]{xy}
\usepackage{rotating}

\usepackage[htt]{hyphenat} % for line breaks in \texttt  

\numberwithin{equation}{section}

\newtheorem*{theorem*}{Theorem}
\newtheorem*{corollary*}{Corollary}

\AfterEndEnvironment{thm}{\noindent\ignorespaces}
\newtheorem{theorem}[equation]{Theorem}
\AfterEndEnvironment{theorem}{\noindent\ignorespaces}

\newtheorem{corollary}[equation]{Corollary}
\AfterEndEnvironment{cor}{\noindent\ignorespaces}
\newtheorem{lemma}[equation]{Lemma}
\AfterEndEnvironment{lemma}{\noindent\ignorespaces}

\AfterEndEnvironment{lem}{\noindent\ignorespaces}
\newtheorem{proposition}[equation]{Proposition}
\AfterEndEnvironment{proposition}{\noindent\ignorespaces}

\AfterEndEnvironment{prop}{\noindent\ignorespaces}
\newtheorem{conjecture}[equation]{Conjecture}
\AfterEndEnvironment{conjecture}{\noindent\ignorespaces}

\AfterEndEnvironment{prob}{\noindent\ignorespaces}

\AfterEndEnvironment{question}{\noindent\ignorespaces}

\theoremstyle{definition}

\newtheorem{definition}[equation]{Definition}
\AfterEndEnvironment{defn}{\noindent\ignorespaces}
\AfterEndEnvironment{definition}{\noindent\ignorespaces}
\newtheorem*{pf_no_qed}{Proof}

\AfterEndEnvironment{pf}{\noindent\ignorespaces}
\newtheorem{eg_no_qed}[equation]{Example}

\AfterEndEnvironment{eg}{\noindent\ignorespaces}

\AfterEndEnvironment{example}{\noindent\ignorespaces}

\AfterEndEnvironment{rmk}{\noindent\ignorespaces}
\newtheorem{remark}[equation]{Remark}
\AfterEndEnvironment{remark}{\noindent\ignorespaces}

\theoremstyle{remark}

\AfterEndEnvironment{claim}{\noindent\ignorespaces}
\newtheorem*{claimpf_no_qed}{Proof of Claim}

\AfterEndEnvironment{claimpf}{\noindent\ignorespaces}

\usepackage{tcolorbox}
\usepackage{tabularx}
\usepackage{colortbl}
\usepackage{booktabs}
\usepackage{adjustbox}
\usepackage{graphbox}
\usepackage{longtable} % for 'longtable' environment
\usepackage{pdflscape} % for 'landscape' environment
\usepackage{adjustbox}
\usepackage{accents}

\usepackage{bm}

\usepackage{autobreak}   

\font\pipefont=lcircle10
\def\elbow{\smash{\raise3pt\hbox{\pipefont\rlap{\rlap{\char'014}\char'016}}}}

\def\halfelbow{\smash{\raise2pt\hbox{\pipefont\rlap{\rlap{\rlap{\char'015}\phantom{\char'017}}}}}}
\def\cross{\smash{\lower5pt\hbox{\rlap{\vrule height16pt}}\raise3pt\hbox{\rlap{\hskip-8pt \vrule height0.4pt depth0pt width16pt}}}}

\def\M{\mathcal{M}}
\def\Q{\mathbf{Q}}

\def\T{\mathcal{T}}
\def\F{\mathcal{F}}

\def\tY{\widetilde{Y}}

\DeclareMathOperator{\NC}{NC}
\DeclareMathOperator{\tree}{tree}
\DeclareMathOperator{\forest}{forest}

\DeclareMathOperator{\DP}{D}

\DeclareMathOperator{\Verts}{\mathcal{V}}
\DeclareMathOperator{\Edges}{\mathcal{E}}
\DeclareMathOperator{\Trees}{Trees}

\DeclareMathOperator{\Gr}{Gr}

\newcommand{\Grk}{\Gr_{k,n}^{\ge 0}}

\newcommand{\R}{\mathbb{R}}

\newcommand{\Z}{\mathbb{Z}}
\newcommand{\QQ}{\mathbb{Q}}

\newcommand{\rf}[1]{\hyperref[#1]{(\ref*{#1})}}

\DeclareMathOperator{\Trop}{Trop}

\newcommand{\lrangle}[1]{\langle #1 \rangle}


\title[Grass trees and forests]
 {Grass(mannian) trees and forests: Variations of the exponential formula, with applications to the momentum amplituhedron}
\author{Robert Moerman}
\address{Department of Physics, Astronmy and Mathematics, University of Hertfordshire}
\email{\href{mailto:r.moerman@herts.ac.uk}{r.moerman@herts.ac.uk}}
\author{Lauren K.\ Williams}
\address{Department of Mathematics, Harvard University}
\email{\href{mailto:williams@math.harvard.edu}{williams@math.harvard.edu}}

\thanks{
LW was partially supported by NSF grants DMS-1854316 and DMS-1854512.}

\begin{document}
\begin{abstract}
The Exponential Formula allows one to enumerate any class of combinatorial objects built by choosing a set of connected components and placing a structure on each connected component which depends only on its size. There are multiple variants of this result, including Speicher's result for noncrossing partitions, as well as  analogues of the Exponential Formula for series-reduced planar trees and forests. In this paper we use these formulae to give generating functions for \emph{contracted Grassmannian trees} and \emph{forests}, certain graphs whose vertices are decorated with a \emph{helicity}.  Along the way we enumerate bipartite planar trees and forests, and we apply our results to enumerate various families of permutations: for example, bipartite planar trees are in bijection with separable permutations. 

It is postulated by Livia Ferro, Tomasz {\L}ukowski and Robert Moerman (2020) that contracted Grassmannian forests are in bijection with boundary strata of the momentum amplituhedron, an object encoding the tree-level S-matrix of maximally supersymmetric Yang-Mills theory. With this assumption, our results give a rank generating function for the boundary strata of the momentum amplituhedron, and imply that the Euler characteristic of the momentum amplituhedron is $1$.
\end{abstract}

\maketitle
\setcounter{tocdepth}{1}
\tableofcontents

\section{Introduction}

In recent years, scattering amplitudes research has motivated the study of \emph{amplituhedra}, which can be viewed as generalizations of polytopes into Grassmannians. There are two amplituhedra relevant to the physics of tree-level particle scattering in maximally supersymmetric Yang-Mills theory: the \emph{(tree) amplituhedron} $\mathcal{A}_{n,k,m}$, introduced by Arkani-Hamed--Trnka in \cite{Arkani-Hamed:2013jha}, and the \emph{momentum amplituhedron} $\mathcal{M}_{n,k,m}$, introduced for $m=4$ by Damgaard--Ferro--{\L}ukowski--Parisi in \cite{Damgaard:2019ztj} and later generalized to any even $m$ by {\L}ukowski--Parisi--Williams in \cite{Lukowski:2020dpn}. For $m=4$, both the amplituhedron and the momentum amplituhedron encode tree-level scattering amplitudes in this theory, but using different \emph{kinematic spaces}: the amplituhedron uses \emph{momentum twistor space} while the momentum amplituhedron uses \emph{momentum space}. Both objects are defined as the image of the \emph{totally nonnegative Grassmannian} $\Gr_{k,n}^{\ge 0}$ under a particular map, where $\Gr_{k,n}^{\ge 0}$ is the subset of the real Grassmannian where all Pl\"{u}cker coordinates are nonnegative, first studied by Postnikov and Lusztig \cite{Postnikov:2006kva,lusztig1994total}.

In addition to their physical significance when $m=4$, these amplituhedra are mathematically interesting objects   which have  been studied in various examples. For $k+m=n$, the amplituhedron is isomorphic to the totally nonnegative Grassmannian, whose rank generating function  was computed in \cite{williams2005enumeration}. When $k=1$, the amplituhedron is a cyclic polytope \cite{Arkani-Hamed:2013jha}, and when $m=1$, the amplituhedron is homeomorphic to the bounded complex of a cyclic hyperplane arrangement \cite{Karp:2016uax}, which also has an explicit rank generating function. When $m=2$, the amplituhedron's boundary strata were classified and enumerated in \cite{Lukowski:2019kqi}. Less is known about the momentum amplituhedron $\mathcal{M}_{n,k,m}$, but for $m=4$, a conjectural characterization of the boundary strata of $\mathcal{M}_{n,k,4}$ (also denoted by $\mathcal{M}_{n,k}$) was given in \cite{Ferro:2020lgp}.  The authors computed the rank generating functions for $\mathcal{M}_{n,k}$ for some small values of $k$ and $n$, and found in those cases that the Euler characteristic was $1$.

The goal of this paper is to enumerate the boundary strata of the $m=4$ momentum amplituhedron $\mathcal{M}_{n,k}$ (for any $k$ and $n$) according to their dimension. To do so, we begin by reformulating the speculative description of boundary strata from  \cite{Ferro:2020lgp} in terms of \emph{Grassmannian forests}, which are acyclic \emph{Grassmannian graphs}. (Grassmannian graphs first appeared implicitly in \cite{Arkani-Hamed:2012zlh} as a generalization of \emph{plabic graphs}, and were subsequently studied in \cite{Postnikov:2018jfq}.) Having described the boundary strata in terms of Grassmannian forests, we proceed to enumerate Grassmannian forests using two variations of the well-known \emph{Exponential Formula}. The first is Speicher's  analogue of the Exponential Formula for non-crossing partitions  \cite{speicher1994multiplicative}. The second variation is an analogue for series-reduced planar trees; it can be viewed naturally in the theory of \emph{species}, and can also be viewed as a combinatorial interpretation of Lagrange Inversion, see \cite{bergeron1998combinatorial,ardila2015algebraic}. Putting together these two variations leads to an analogue of the Exponential Formula for series-reduced planar forests, which we use to enumerate contracted Grassmannian forests according to \emph{helicity} and \emph{momentum amplituhedron dimension}.  Along the way, we also enumerate \emph{contracted plabic trees and forests}, which are in bijection with bipartite planar trees and forests. And we can translate all of our results into enumerative results about various class of permutations: for example, contracted plabic trees are in bijection with separable permutations.

The paper is structured as follows. In \cref{sec:variation}, we review the Lagrange Inversion formula and Speicher's noncrossing partition analogue of the Exponential Formula.  We then give series-reduced planar tree and forest analogues of the Exponential Formula. In \cref{sec:Grassmannian}, we introduce the totally nonnegative Grassmannian, as well as Grassmannian graphs, trees, and forests. We enumerate contracted plabic trees and forests as a warmup, keeping track of their \emph{helicity} $k$, number of boundary vertices $n$, and their \emph{momentum amplituhedron dimension}. We then enumerate contracted Grassmannian trees and forests according to the same statistics. In \cref{sec:perm} we give applications to enumeration of permutations. In \cref{sec:momentum} we define the momentum amplituhedron $\mathcal{M}_{n,k}$ and interpret our combinatorial results as enumerating the boundary strata of $\mathcal{M}_{n,k}$. By specializing $q=-1$ in our generating function, we also find that the Euler characteristic equals $1$. We emphasise that the results of \cref{sec:momentum} assume the characterization of momentum amplituhedron boundaries conjectured in \cite{Ferro:2020lgp}. We end the paper with an appendix providing a table of rank generating functions for the momentum amplituhedron $\mathcal{M}_{n,k}$ for various values of $n$ and $k$.

%%%%%%%%%%%%%%%%%%%%%%%

\section{Variations of the Exponential Formula}\label{sec:variation}

Many combinatorial objects can be built by choosing a set of connected components, then placing some structure on each connected component. In this case, if one understands how to enumerate the ways of placing a structure on each connected component of a given size, then the well-known \emph{Exponential Formula} (see e.g.\ \cite{stanley_ec2,bergeron1998combinatorial,ardila2015algebraic}) allows one to enumerate the combinatorial objects. There is an analogue of the Exponential Formula, due to Speicher \cite{speicher1994multiplicative}, for combinatorial objects which can be built by choosing a noncrossing partition and then placing a structure on each block of the noncrossing partition.  In this section we will provide background on Lagrange Inversion and Speicher's result, then give analogues of the Exponential Formula for series-reduced planar trees and forests.

\subsection{Lagrange Inversion}
The set $xK[[x]]$ of all formal power series $a_1 x+a_2 x^2+ \dots $ with zero constant term over a field $K$ forms a monoid under the operation of functional composition. The identity element of this monoid is the power series $x$.
\begin{definition}
	If $f(x) = a_1 x+ a_2 x^2+ \dots \in K[[x]]$, then we call a power series $g(x)$ a \emph{compositional inverse} of $f$ if $f(g(x)) = g(f(x))=x$, in which case we write $g(x) = f^{\lrangle{-1}}(x)$.  
\end{definition}

\begin{theorem} [Lagrange inversion formula, {\cite[Theorem 5.4.2]{stanley_ec2}}]\label{thm:LIF}
	Let $K$ be a field with $\text{char}\,K = 0$, and let $C(x)\in x K[[x]]$ with $[x]C(x)\ne0$. Then for positive integers $k, n$ we have
	$$n [x^n] C^{\lrangle{-1}}(x)^{k}=k[x^{n-k}]\left(\frac{x}{C(x)}\right)^n.$$
\end{theorem}

\subsection{Speicher's noncrossing partition analogue of the Exponential Formula}

\begin{definition}
	A \emph{noncrossing partition} of $[n]$ is a partition  $\pi$ of the set $[n]=\{1,2,\dots,n\}$ into blocks $B_1,\dots, B_{\ell}$ satisfying the following condition: if $a<b<c<d$ and $B$ and $B'$ are blocks of $\pi$ such that $a,c\in B$ and $b,d\in B'$, then $B=B'$.
\end{definition}

Equivalently, $\pi$ is a noncrossing partition if after drawing the numbers $1,2,\dots,n$ in order around a circle, and replacing each block $B$ with the convex hull of the corresponding points, the resulting polygons do not overlap. See \cref{fig:nc9}.

\begin{figure}[h]
	\centering
	\includegraphics[scale=0.4]{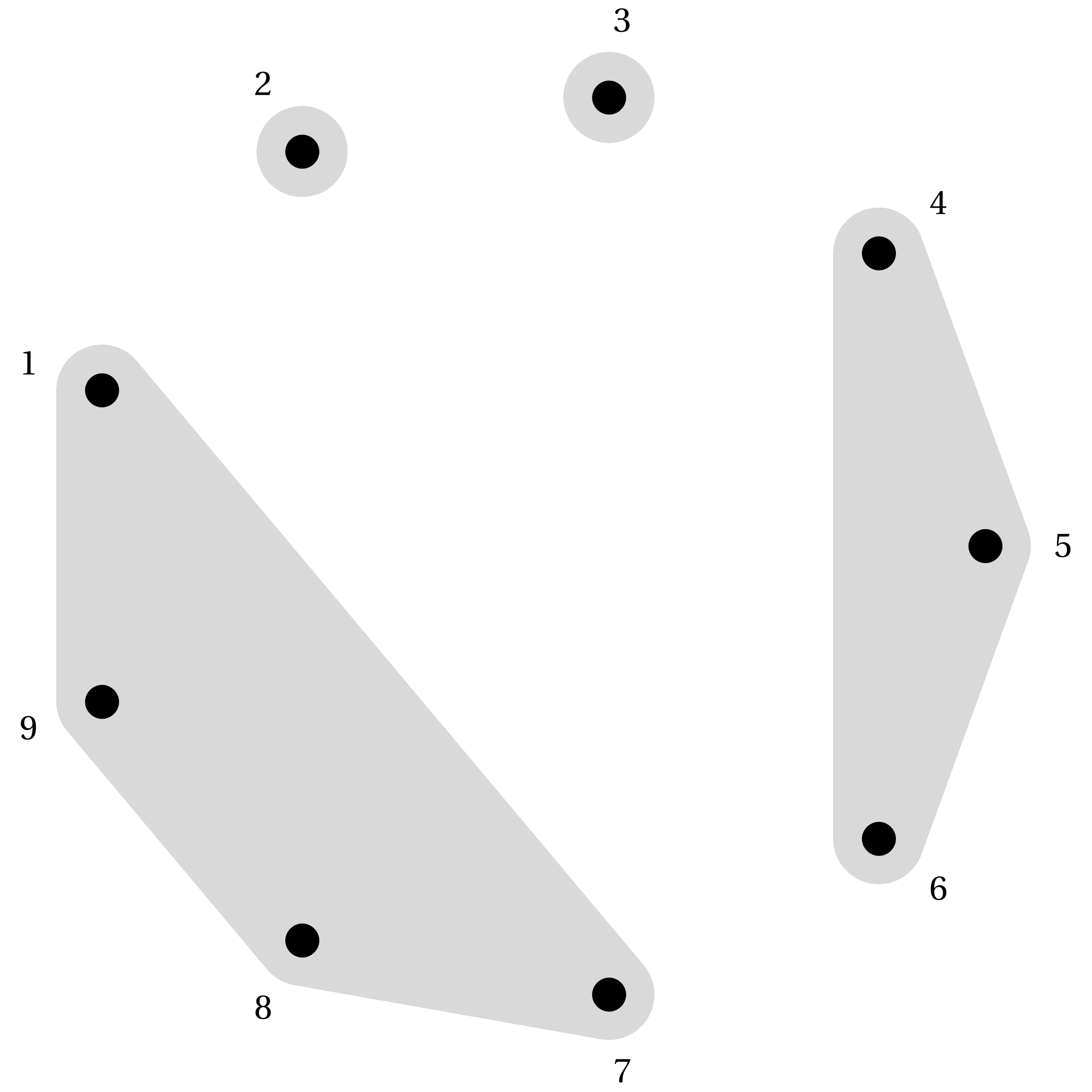}
	\caption{The non-crossing partition $\{\{1,7,8,9\},\{2\},\{3\},\{4,5,6\}\}$ of $[9]$ with each block drawn as a convex hull of the corresponding points.}
	\label{fig:nc9}
\end{figure}

Let $\NC_n$ denote the lattice of noncrossing partitions of $[n]$. The following result says that if a class $\mathcal{H}_{\NC}$ of combinatorial objects is built by choosing a noncrossing partition and putting a structure on each block independently (encoded by the function $f$), then the generating function $H_{\NC}(x)$ for $\mathcal{H}_{\NC}$ can be obtained from the generating function $F(x)$ for $f$.

\begin{theorem} [{\cite{speicher1994multiplicative}, see also \cite[Exercise 5.35b]{stanley_ec2}}] 
	\label{thm:speicher}
	Let $K$ be a field. Given a function $f:\Z^+ \to K$, define	a new function $h:\Z^+ \to K$ by 
	$$h(n)=\sum_{\pi = \{B_1,\dots,B_\ell\}\in \NC_n} f(\#B_1) f(\#B_2)\dots f(\# B_\ell),$$
	where $\#B_i$ denotes the cardinality of block $B_i$. Let $F(x)=1+\sum_{n\geq 1} f(n)x^n$ and $H_{\NC}(x) = 1+\sum_{n \geq 1} h(n)x^n$. Then 
	$$x H_{\NC}(x) = {\left( \frac{x}{F(x)} \right)}^{\lrangle{-1}}.$$	
\end{theorem}

\subsection{Series-reduced planar tree and  forest analogues of the Exponential Formula}
In this section we give \cref{thm:tree} and \cref{thm:forest} which can be viewed as analogues of Speicher's Theorem, but with series-reduced planar trees and planar forests replacing noncrossing partitions.  We note that the result on series-reduced planar trees can be viewed as an analogue of Speicher's Theorem for polygon dissections.

\begin{definition}
	A \emph{planar tree} $T$ (respectively, a \emph{planar forest} $F$) on $n$ leaves is a tree (resp., a forest) properly embedded in a disk with $n$ \emph{boundary vertices} (i.e.\ vertices of degree $1$) on the boundary of the disk (labelled in clockwise order). Let $\Verts_\text{int}(T)$ and $\Verts_\text{int}(F)$ denote the set of internal vertices (i.e.\ vertices with degree at least $2$) of $T$ and $F$.
	
	A planar tree or forest is called \emph{series-reduced} if it has no internal vertices of degree $2$. Let $\T_n$ (resp., $\F_n$) denote the set of series-reduced planar trees (resp., forests) on $n$ leaves. (The requirement on internal vertices implies that $\T_n$ and $\F_n$ are finite.) A series-reduced planar tree $T$ on $n$ leaves is said to be of \emph{type} $\bm{r}=(r_3,\ldots,r_n)$ if it has $r_i$ internal vertices of degree $i$. An example of a series-reduced planar tree on $9$ leaves of type $(2,1,1,0,0,0,0)$ is given in \cref{fig:tree-9} (middle). Let $\T_n(r_3,\ldots,r_n)$ denote the subset of $\T_n$ of type $(r_3,\ldots,r_n)$.
\end{definition}

\begin{definition}
	A \emph{plane tree} on $n-1$ leaves is a rooted tree $\tilde{T}$ with $n-1$ \emph{boundary vertices} (i.e.\ vertices with no descendants) where each internal vertex of $\tilde{T}$ has at least $1$ descendant.
	
	A \emph{Schr\"{o}der tree} is a plane tree which is \emph{series-reduced}, i.e.\ each internal vertex has at least two descendants. A Schr\"{o}der tree $\tilde{T}$ on $n-1$ leaves is said to be of \emph{type} $(r_3,\ldots,r_n)$ if it has $r_i$ internal vertices with $i-1$ descendants. An example of a Schr\"{o}der tree on $8$ leaves of \emph{type} $(2,1,1,0,0,0,0)$ is given in \cref{fig:tree-9} (left). Let $\tilde{\T}_{n-1}$ denote the set of Schr\"{o}der trees on $n-1$ leaves, and let $\tilde{\T}_{n-1}(r_3,\ldots,r_n)$ denote the subset of $\tilde{\T}_{n-1}$ of type $(r_3,\ldots,r_n)$.
\end{definition}

\begin{remark}\label{remark:tree-bijection}
	Given a series-reduced planar tree $T\in \T_n$, let $v_\ast$ be the internal vertex incident to the leaf labelled $1$. If we remove leaf $1$ and its incident edge, then the remaining tree $\tilde{T}$ can be thought of as a Schr\"{o}der tree with root vertex $v_\ast$. This map gives a bijection between the series-reduced planar trees $T \in \T_n(r_3,\ldots,r_n)$ and Schr\"{o}der trees $\tilde{T} \in \tilde{\T}_{n-1}(r_3,\ldots,r_n)$. Note that an internal vertex of degree $d$ in a series-reduced planar tree corresponds to an internal vertex with $d-1$ descendants in a Schr\"{o}der tree.
\end{remark}

\begin{lemma}\label{thm:tree-type-count}
	Let $t_n(r_3,\dots,r_n) =|\T_n(r_3,\ldots,r_n)|$ be the number of series-reduced planar trees of type $\bm{r}=(r_3,\ldots,r_n)$ and let $|\bm{r}| = \sum_i r_i$. Then $$t_n(r_3,\ldots,r_n) = \frac{(n+|\bm{r}|-2)!}{(n-1)! r_3! \cdots r_n!}.$$
\end{lemma}
\begin{proof}
	From \cref{remark:tree-bijection}, series-reduced planar trees of type $(r_3,\ldots,r_n)$ are in bijection with the plane trees of type $(n-1, 0, r_3,\ldots,r_n)$ considered in \cite[Section 5.3]{stanley_ec2}. The formula in the lemma now follows from \cite[Theorem 5.3.10]{stanley_ec2}.
\end{proof}

\begin{figure}[h]
	\centering
	\includegraphics[width=0.3\textwidth]{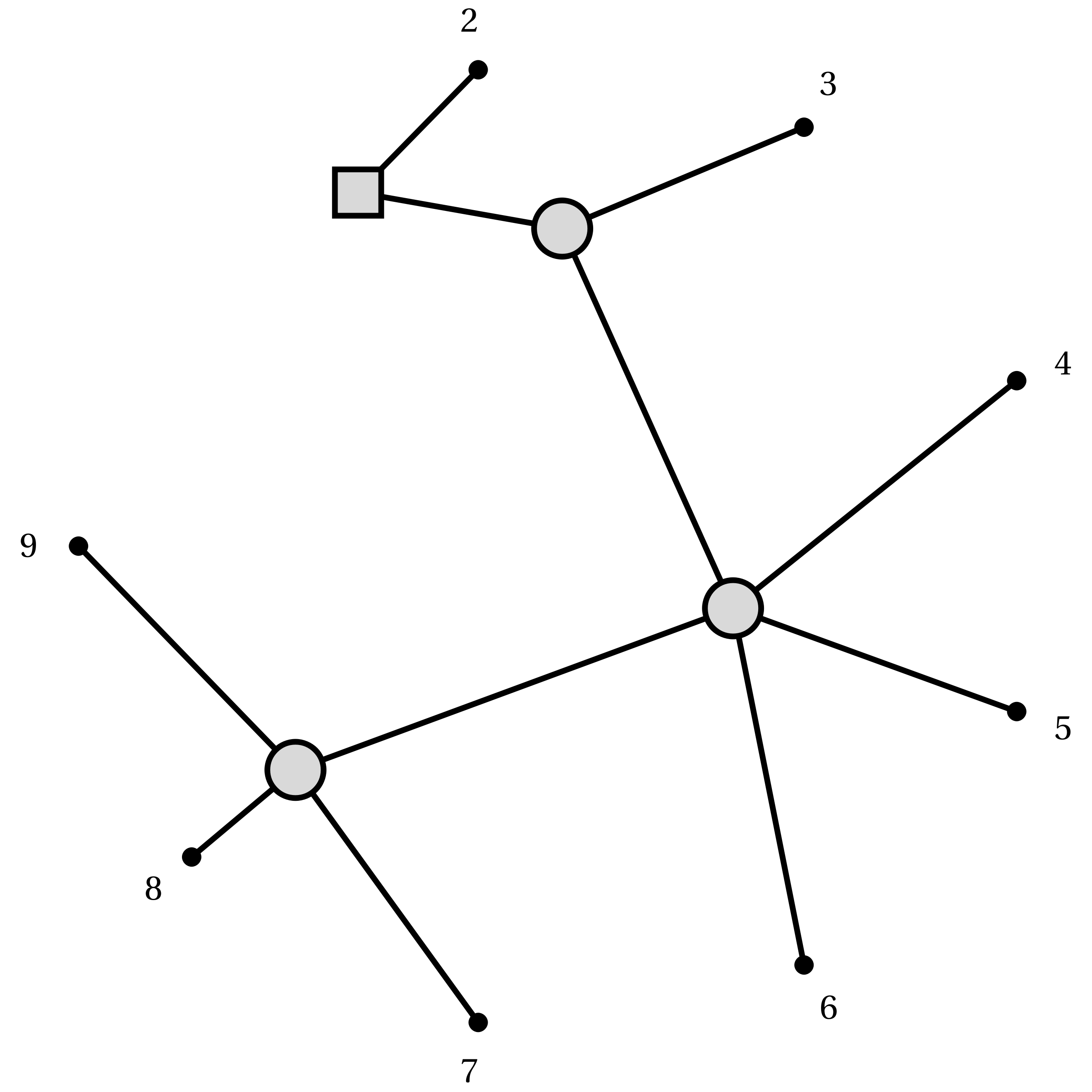}
	\includegraphics[width=0.3\textwidth]{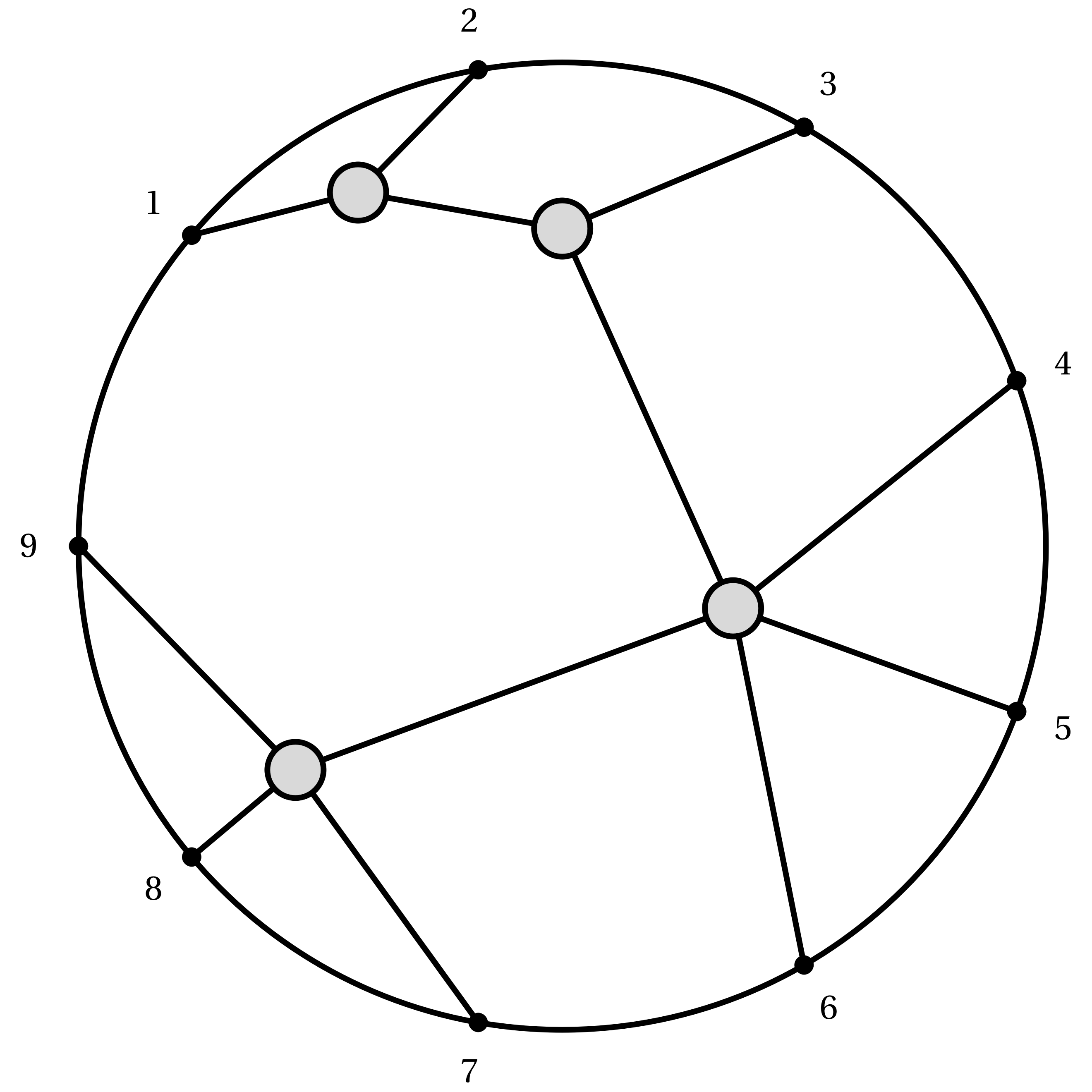}
	\includegraphics[width=0.3\textwidth]{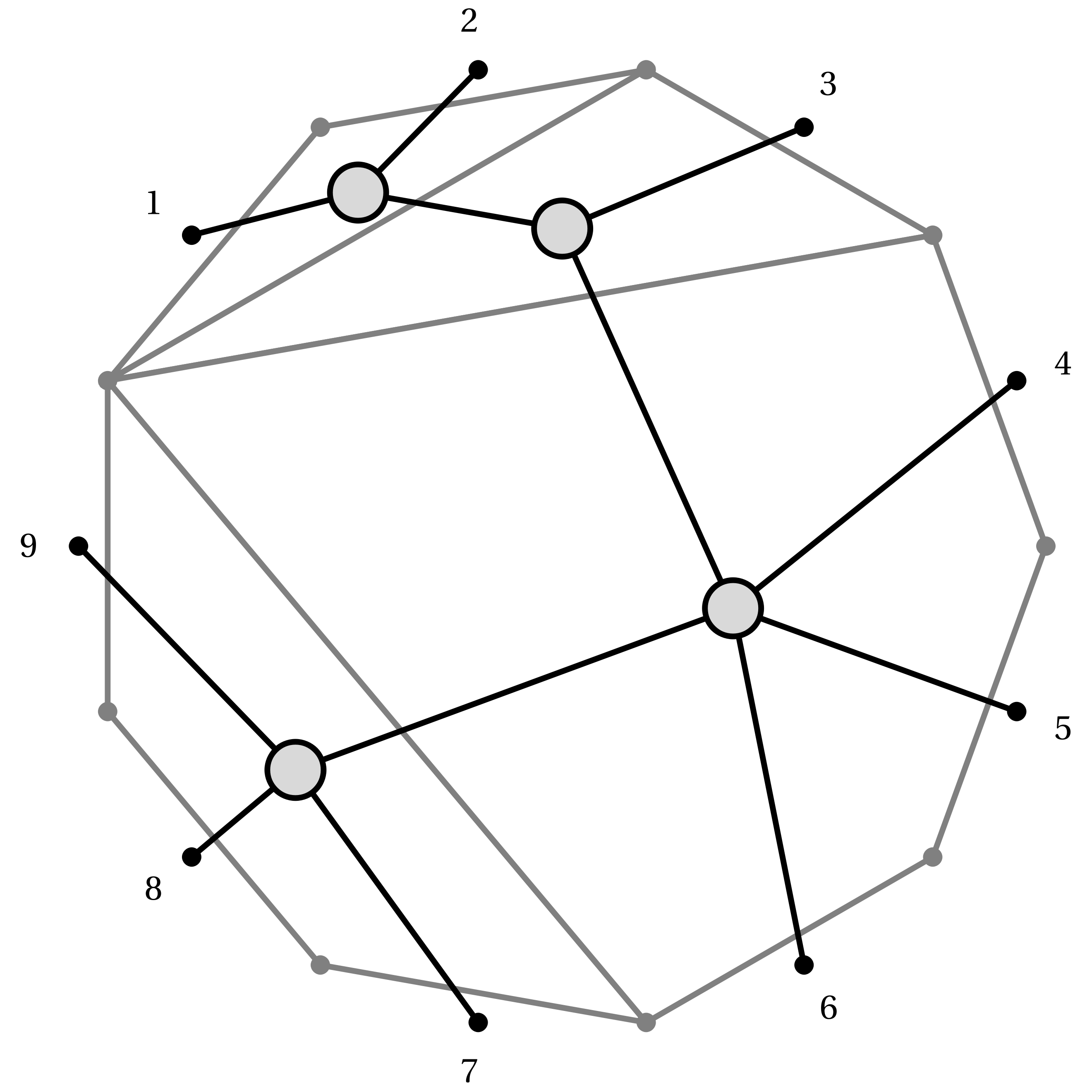}
	
	\caption{A series-reduced planar tree on $9$ leaves of type $(2,1,1,0,0,0,0,0)$ (middle), its corresponding Schr\"{o}der tree on $8$ leaves with the root vertex depicted as a square (left), and its dual polygon dissection of $\Q_9$ (right).}
	\label{fig:tree-9}
\end{figure}

The following result says that if a class $\mathcal{H}_{\tree}$ of combinatorial objects is built by choosing a series-reduced planar tree in $\T_n$ and putting a structure on each internal vertex independently (depending only on its degree and encoded by the function $f$), then the generating function $H_{\tree}(x)$ for $\mathcal{H}_{\tree}$  can be obtained from the generating function $F(x)$ for the function $f$. We note that \cref{thm:tree} has appeared in various references: for example, it fits naturally into the theory of \emph{species} and is closely related to \cite[page 168, Equation (18)]{bergeron1998combinatorial}; it can also be reformulated using Schr\"{o}der trees and viewed as a combinatorial interpretation of Lagrange Inversion, see \cite[Theorem 2.2.1]{ardila2015algebraic}. It is similar in spirit to \cite[Proposition 5]{Morales}.

\begin{theorem}  
	\label{thm:tree}
	Let $K$ be a field.  Given a function $f:\Z_{\ge 3} \to K$, define a new function $h:\Z_{\ge 3} \to K$ by 
	\begin{equation} \label{eq:tree-h}
		h(n)=\sum_{T \in \T_n} \prod_{v\in \Verts_\text{int}(T)} f(\deg(v))\,.
	\end{equation}
	Let $F(x)=\sum_{n\geq 3} f(n)x^n$ and $H_{\tree}(x) = x^2+ \sum_{n \geq 3} h(n)x^n$. Then
	\begin{equation} 
		\label{eq:tree-H}
		\frac{1}{x}H_{\tree}(x) = {\left(x-\frac{1}{x}F(x) \right)}^{\lrangle{-1}}.
	\end{equation}
\end{theorem}

\begin{remark}
	\label{Htree}
	The appearance of the term $x^2$ in $H_{\tree}(x)$  reflects the fact that there is a unique series-reduced planar tree  on $2$ leaves.	If we want to account for the unique tree in $\T_1$ with $1$ leaf and $1$ internal vertex, we can consider $\hat{H}_{\tree}(x) = h(1)x +  H_{\tree}(x)$, where $h(1)=f(1)$ is the number of structures that we can put on an internal vertex of degree $1$.	
\end{remark}

\cref{thm:tree} can be proved in several ways.
\begin{enumerate}
	\item Using \cref{thm:tree-type-count}, one can give an argument analogous to the proof of Speicher's Theorem as presented in \cite[Exercise 5.35b]{stanley_ec2}. We leave this as an exercise. 
	\item One can reformulate this result as a \emph{functional equation} describing the recursive  structure of trees. This can be verified directly but also fits naturally into Joyal's framework of \emph{species} as in \cite{bergeron1998combinatorial}. We thank Ira Gessel and Francois Bergeron for their comments on this approach; see also \cite[Theorem 2.2.1]{ardila2015algebraic}). For completeness, we include this second proof below.
\end{enumerate}

\begin{proof}[Proof of \cref{thm:tree}]	
	Let $\tilde{F}(x)=\frac 1x F(x)=\sum_{n\ge3} f(n)x^{n-1}$ and $\tilde{H}(x)=\frac 1x H(x)=x+\sum_{n\ge3} h(n)x^{n-1}$. That is, the coefficient of $x^{n-1}$ in $\tilde{F}(x)$ gives the number of ways of decorating an internal vertex $v$ in a Schr\"{o}der tree where $v$ has $n-1$ descendants, and the coefficient of $x^{n-1}$ in $\tilde{H}(x)$ enumerates the decorated Schr\"{o}der trees on $n-1$ leaves.
	
	Schr\"{o}der trees can be built by choosing a root vertex $v_\ast$ with degree $d-1$ (where $d\ge 3$) and placing Schr\"{o}der trees $\tilde{T}_1,\ldots,\tilde{T}_{d-1}$ at each of the $d-1$ vertices below $v_\ast$. Thus, it immediately follows that
	\begin{align*}
		\tilde{H}(x) = x + \sum_{n\ge3} f(n)\tilde{H}(x)^{n-1} = x + \tilde{F} (\tilde{H}(x))\,, 
	\end{align*}
	which is equivalent to $\tilde{H}(x) = (x-\tilde{F}(x))^{\lrangle{-1}}$ or $\frac 1xH(x) = (x-\frac 1xF(x))^{\lrangle{-1}}$.
\end{proof}

We now combine \cref{thm:tree} and \cref{thm:speicher} to obtain \cref{thm:forest} below. As before, we interpret this result as saying that if a class $\mathcal{H}_{\forest}$ of combinatorial objects is built by choosing a series-reduced planar forest in $\F_n$ and putting a structure on each internal vertex independently (depending only on its degree and encoded by the function $g$), then the generating function $H_{\forest}(x)$ for $\mathcal{H}_{\forest}$ can be obtained from the generating function $G(x)$ for the function $g$. 

\begin{corollary}\label{thm:forest}
	Let $K$ be a field.  Given a function $g:\Z^+ \to K$, define a new function $h:\Z^+ \to K$ by 
	\begin{equation} \label{eq:forest-h}
		h(n)=\sum_{F \in \F_n} \prod_{v\in \Verts_\text{int}(F)} g(\deg(v)).
	\end{equation}
	(Note that $h(1)=g(1)$, and that $g(2)$ is irrelevant for series-rooted forests.)
	
	Let $G(x)=\sum_{n\geq 3} g(n)x^n$ and $H_{\forest}(x) = 1+\sum_{n \geq 1} h(n)x^n$. Then
	\begin{equation} \label{eq:forest-H}
		x H_{\forest}(x) = 
		{\left( \frac{x}{1+\hat{H}_{\tree}(x)} \right)}^{\lrangle{-1}} = 
		{\left( \frac{x}{1+h(1)x+ x\left( x-\frac{G(x)}{x} \right)^{\lrangle{-1}} } \right)}^{\lrangle{-1}}, 
	\end{equation}
	where $\hat{H}_{\tree}(x)$ is given by \cref{Htree} and \eqref{eq:tree-H}.
\end{corollary}

\begin{proof}
	The connected components of a planar forest embedded in a disk have the structure of a noncrossing partition.  So we can enumerate planar forests by applying \cref{thm:speicher}, with $1+\hat{H}_{\tree}(x)$ playing the role of $F(x)$ (which has constant term $1$).  The result now follows from \cref{thm:tree}.
\end{proof}

Although we don't need it for what follows, we note that the trees in $\T_n$ are in bijection with \emph{dissections of a polygon}.

\begin{definition}
	Let $\mathbf{Q}_n$ denote a convex $n$-gon with vertices labelled $1,2,\dots,n$. A \emph{dissection} of $\Q_n$ is a subdivision $\rho$ of $\Q_n$ into smaller polygons $\{P_1,\dots, P_{\ell}\}$ obtained by drawing diagonals that don't intersect in their interiors. We say that a dissection of $\Q_n$ has \emph{type} $(m_3,\dots,m_n)$ if it consists of $m_i$ $i$-gons. \cref{fig:tree-9} (right) gives an example of a dissection of $\Q_9$ of type $(2,1,1,0,0,0,0)$.
\end{definition}

Clearly the series-reduced planar trees in $\T_n(m_3,\dots,m_n)$ are also in bijection with the dissections of $\Q_n$ of type $(m_3,\dots,m_n)$. Both of these objects are in bijection with the faces of the $(n-3)$-dimensional associahedron as well as its normal fan. The latter is combinatorially the Stanley--Pitman fan \cite{stanley2002polytope} or the tropical totally positive Grassmannian $\Trop^+\Gr_{2,n}$ \cite{speyer2005tropical}.

We can now translate \cref{thm:tree} into the language of polygon dissections. Let $\DP_n$ denote the set of dissections of 
the $n$-gon $\Q_n$.  The following theorem follows from the Implicit Species Theorem of Joyal \cite{bergeron1998combinatorial} and also appeared in \cite[Theorem 4.2]{schuetz2016polygonal}.

\begin{theorem} 
	\label{thm:DP}
	Let $K$ be a field.  Given a function $f:\Z_{\ge 3} \to K$, define a new function $h:\Z_{\ge 3} \to K$ by
	$$h(n)=\sum_{\rho = \{P_1,\dots,P_\ell\}\in \DP_n} f(\#P_1) f(\#P_2)\dots f(\# P_\ell),$$ 
	where $\#P_i$ denotes the number of vertices in the polygon $P_i$. Let $F(x)=\sum_{n\geq 3} f(n)x^n$ and $H(x) = x^2+ \sum_{n \geq 3} h(n)x^n$. Then
	\begin{equation*}
		\frac{1}{x}H(x) = {\left(x-\frac{1}{x}F(x) \right)}^{\lrangle{-1}}.
	\end{equation*}
\end{theorem}

%%%%%%%%%%%%%%%%%%%%%%%

\section{Grassmannian graphs, trees and forests}\label{sec:Grassmannian}

The totally nonnegative Grassmannian, a particular semi-algebraic subset of the real Grassmannian, was first introduced by Postnikov and Lusztig \cite{Postnikov:2006kva,lusztig1994total}, and it has been a topic of intense investigation by both mathematicians and physicists. In this section we introduce the totally nonnegative Grassmannian together with (equivalence classes of) Grassmannian graphs, which index its cells. We also introduce Grassmannian \emph{forests}, which (conjecturally) enumerate the boundary strata of a related object called the \emph{momentum amplituhedron.}  We  then provide explicit enumeration formulae for Grassmannian trees and forests.

\subsection{The totally nonnegative Grassmannian and its positroid stratification}

Fix integers $k,n$ such that $0\le k\le n$. The \emph{Grassmannian} $\Gr_{k,n}(\mathbb{F})$ over a field $\mathbb{F}$ is the variety of all $k$-dimensional subspaces of $\mathbb{F}^n$. Each element of $\Gr_{k,n}(\mathbb{F})$ may be represented by a full rank $k\times n$ matrix, modulo  row operations, whose rows span the $k$-dimensional subspace.  We denote the element of $\Gr_{k,n}(\mathbb{F})$ represented by the matrix $C$ by $[C]$. 

In what follows, let $\Gr_{k,n}=\Gr_{k,n}(\mathbb{R})$ denote the \emph{real Grassmannian}. Recall that $[n]:=\{1,2,\ldots,n\}$. Let $\binom{[n]}{k}$ denote the set of all $k$-element subsets of $[n]$. Given an element $V=[C]\in\Gr_{k,n}$, the maximal minors $\Delta_I(C)$ of $C$, where $I\in\binom{[n]}{k}$, form projective coordinates on $\Gr_{k,n}$, called the \emph{Pl\"{u}cker coordinates} of $V$. Consequently, we write $\Delta_I(V)$ to denote $\Delta_I(C)$.

\begin{definition}[{\cite[Section 3]{Postnikov:2006kva}}] 
	An element of $\Gr_{k,n}$ is said to be \emph{totally positive} (resp., \emph{totally nonnegative}) if all of its Pl\"{u}cker coordinates are positive (resp., nonnegative). The \emph{totally positive Grassmannian} $\Gr_{k,n}^{> 0}$ (resp., \emph{totally nonnegative Grassmannian} $\Gr_{k,n}^{\ge 0}$) is the semi-algebraic subset of all totally positive (resp., totally nonnegative) elements of $\Gr_{k,n}$. For each $M\subset\binom{[n]}{k}$, let $S_M$ be the subset of elements of $\Gr_{k,n}^{\ge 0}$ whose Pl\"{u}cker coordinates are all strictly positive for $I\in M$ and otherwise zero. If $S_M$ is non-empty, $M$ is called a \emph{positroid} and $S_M$ its \emph{positroid cell}. 
\end{definition}

Each positroid cell is indeed a topological cell \cite[Theorem 6.5]{Postnikov:2006kva}, and moreover, the positroid cells of $\Gr_{k,n}^{\ge 0}$ glue together to form a CW complex \cite{postnikov2009matching}.

It was shown in \cite{Postnikov:2006kva} that positroid cells of $\Gr_{k,n}$ are in bijection with various combinatorial objects including (equivalence classes of) \emph{reduced plabic graphs} $G$ of type $(k,n)$, and \emph{decorated permutations} $\sigma$ on $[n]$ with $k$ anti-excedances (as will be defined in \cref{sec:perm}). Consequently, these objects provide unambiguous labels for positroid cells; we will write $S_G$ or $S_\sigma$ to denote the positroid cell associated to $G$ or $\sigma$, respectively. 

\subsection{Grassmannian graphs, trees, and forests}

The following notion of \emph{Grassmannian graph} implicitly appeared in \cite{Arkani-Hamed:2012zlh} as a generalization of plabic graphs; the definition below then formally appeared in \cite{Postnikov:2018jfq}.

\begin{definition}[{\cite[Definition 4.1]{Postnikov:2018jfq}}]
	A \emph{Grassmannian graph}\footnote{Our definition differs from Postnikov's \cite[Definition 4.1]{Postnikov:2018jfq} in two ways: (1) we exclude vertices of degree $2$ and (2) unless $v$ is a boundary leaf, we do not allow the helicity of a vertex $v$ to be $0$ or $\deg(v)$. These further restrictions guarantee that our Grassmannian graphs which are forests are always reduced in the sense of \cite[Definition 4.5]{Postnikov:2018jfq}.} is a finite planar graph $G$ with \emph{vertices} $\Verts(G)$ and \emph{edges} $\Edges(G)$, embedded in a disk, with $n$ \emph{boundary vertices} $\{b_1,b_2,\ldots,b_n\}$ of degree $1$ on the boundary of the disk (labelled in  clockwise order). We let $\Verts_\text{int}(G) = \Verts(G)\setminus\{b_1,b_2,\ldots,b_n\}$ denote the vertices in the interior of the disk, each of which is required to be connected by a path to the boundary of the disk. We require moreover that $G$ has no vertices of degree $2$. Each $v\in\Verts_\text{int}(G)$ is given a nonnegative integral \emph{helicity} $h(v)$: if $v$ is a \emph{boundary leaf}, i.e.\ an internal vertex of degree $1$ connected to a boundary vertex, then $h(v) = 0$ or $1$; otherwise, $1 \leq h(v) \leq \deg(v)-1$, where $\deg(v)$ denotes the degree of $v$. We say that $v$ has \emph{type} $(k,d)$ when $k=h(v)$ and $d=\deg(v)$. We refer to $v$ as being \emph{white} if $h(v)=1$ and  \emph{black} if $h(v)=\deg(v)-1$; otherwise we call it \emph{generic}.
	
	The \emph{helicity} of a Grassmannian graph $G$ with $n$ boundary vertices is the number $h(G)$ given by 
	\begin{align}\label{eq:helicity-G}
		h(G):=\sum_{v\in\Verts_\text{int}(G)}\left(h(v)-\frac{\deg(v)}{2}\right) +\frac{n}{2}\, \in \{0,1,\dots, n\}.
	\end{align}  
	Such a Grassmannian graph $G$ is said to be of \emph{type} $(k,n)$ where $k=h(G)$.
	
	A \emph{plabic graph} is a Grassmannian graph in which each internal vertex is either white
	or black; that is, there are no generic vertices.
\end{definition}

In this article we will restrict our attention to Grassmannian graphs which are forests, in other words, have no internal cycles.

\begin{definition}
	A \emph{plabic forest} is an acyclic plabic graph, and a \emph{plabic tree} is a connected acyclic plabic graph. Similarly, a \emph{Grassmannian forest}  is an acyclic Grassmannian graph, and a \emph{Grassmannian tree} is a connected acyclic Grassmannian graph. If $F$ is a Grassmannian forest, we let $\Trees(F)$ denote the Grassmannian trees in $F$.
\end{definition}

\begin{remark}
	Given a Grassmannian tree $T$ with $n$ boundary vertices,
	\begin{align}\label{eq:multiplicity-T}
		n=\sum_{v\in\Verts_\text{int}(T)}\deg(v)-2(|\Verts_\text{int}(T)|-1)=2+\sum_{v\in\Verts_\text{int}(T)}(\deg(v)-2)\,,
	\end{align}
	and substituting this into \eqref{eq:helicity-G}, we find that its helicity can be expressed as
	\begin{align}\label{eq:helicity-T}
		h(T)=\sum_{v\in\Verts_\text{int}(T)}h(v)-(|\Verts_\text{int}(T)|-1)=1+\sum_{v\in\Verts_\text{int}(T)}(h(v)-1)\,.
	\end{align}
\end{remark}

\begin{remark}\label{remark:tree-non-generic}
	Every internal vertex of a Grassmannian tree $T$ of type $(1,n)$ (resp.\ $(n-1,n)$) must be white (resp.\ black). The observation follows from \eqref{eq:helicity-T} (and \eqref{eq:multiplicity-T}) together with the restriction that $h(v)\ge 1$ (resp.\ $h(v)\le \deg(v)-1$) for every internal vertex $v$.
\end{remark}

There is a natural partial order and equivalence relation which can be defined for Grassmannian forests.

\begin{definition}\label{def:contracted}
	Given Grassmannian forests $F$ and $F'$, we say that $F'$ \emph{coarsens} $F$ (and that $F$ \emph{refines} $F'$) if $F'$ can be obtained from $F$ by applying a sequence of \emph{vertex contraction moves} in which two adjacent internal white vertices (or two adjacent internal black vertices) get contracted into a single white (or black) vertex (see \cref{fig:contraction-move}). The \emph{refinement order} on Grassmannian forests is the partial order $\le_\text{ref}$ where $F\le_\text{ref} F'$ if $F'$ coarsens $F$. Grassmannian forests $F$ and $F'$ are \emph{refinement-equivalent} if one can be obtained from the other by a sequence of refinements and coarsenings. A Grassmannian forest is said to be \emph{contracted} or \emph{maximal} if it is maximal with respect to the refinement order.
\end{definition}

\begin{figure}[h]
	\centering
	\begin{minipage}{.5\textwidth}
		\begin{align*}
			\vcenter{\hbox{\includegraphics[scale=0.3]{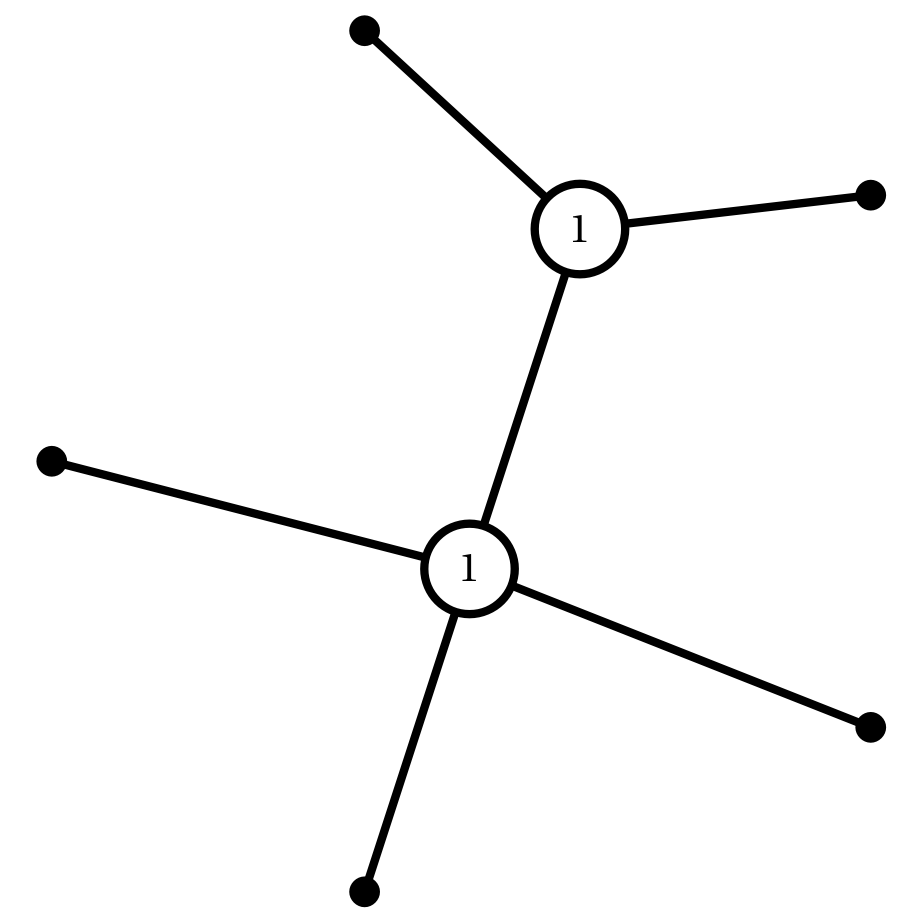}}}\to\vcenter{\hbox{\includegraphics[scale=0.3]{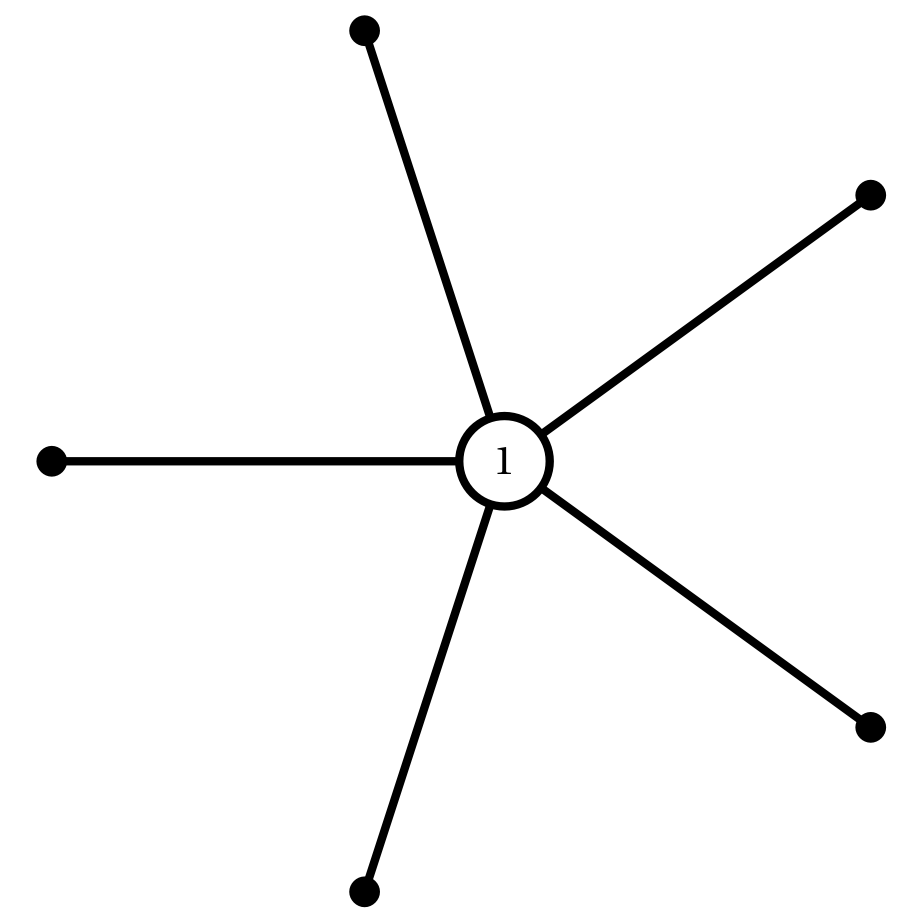}}}
		\end{align*}
	\end{minipage}%
	\begin{minipage}{.5\textwidth}
		\begin{align*}
			\vcenter{\hbox{\includegraphics[scale=0.3]{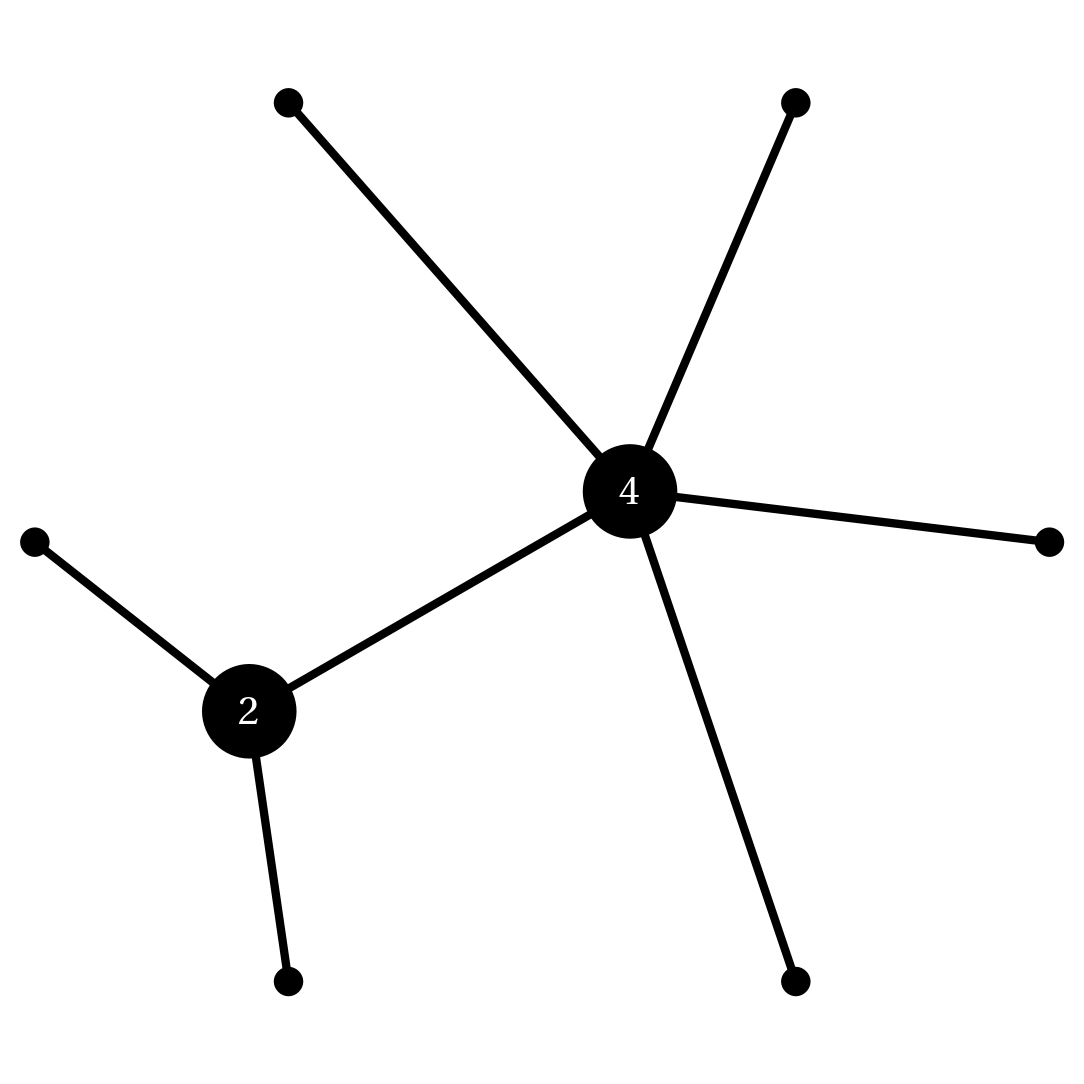}}}\to\vcenter{\hbox{\includegraphics[scale=0.3]{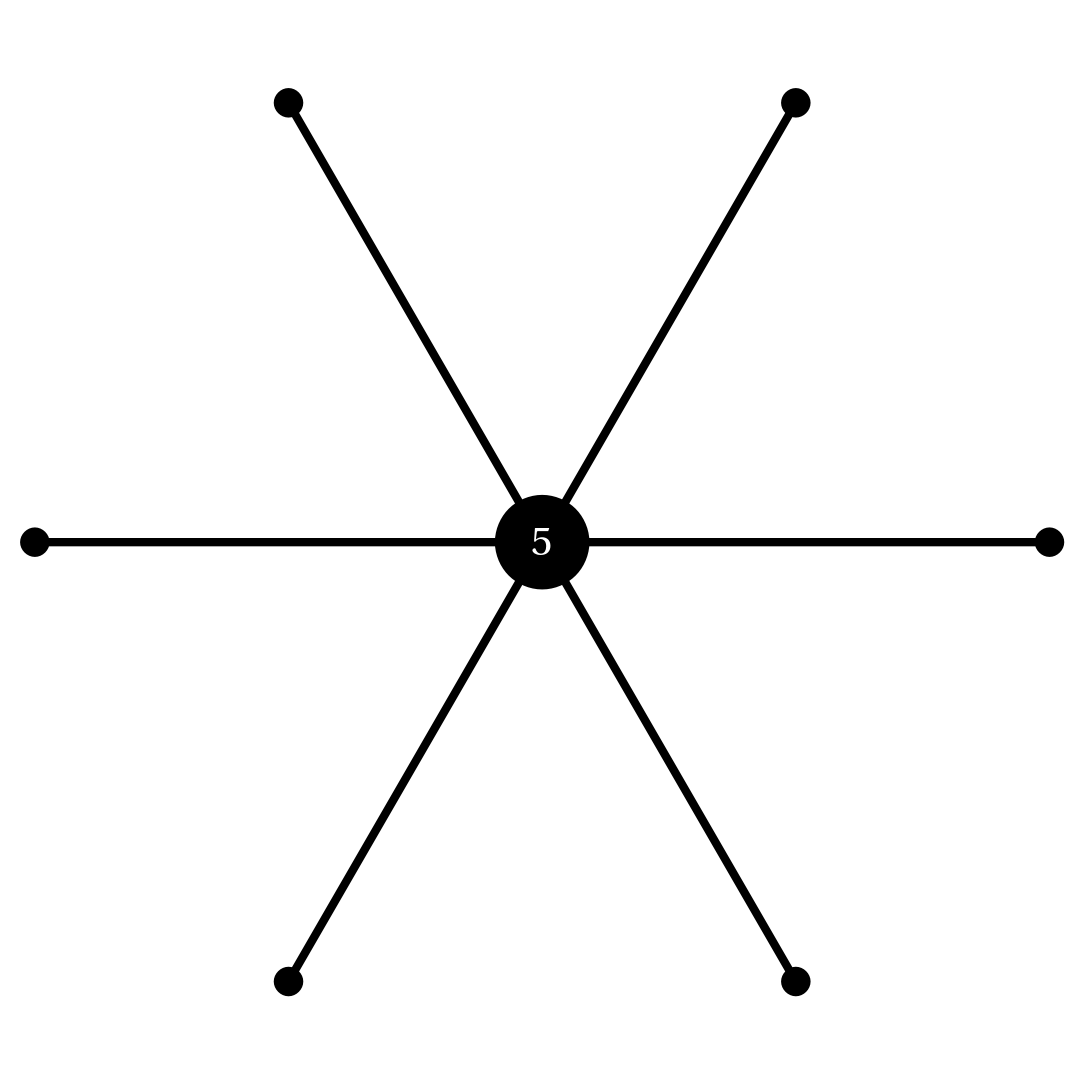}}}
		\end{align*}
	\end{minipage}
	\caption{Examples of vertex contraction moves.}
	\label{fig:contraction-move}
\end{figure}

More generally, there is a notion of refinement order for Grassmannian graphs \cite[Definition 4.7]{Postnikov:2018jfq} which coincides with the above definition when restricted to Grassmannian forests. Positroid cells of $\Grk$ are also in bijection with refinement-equivalence classes of reduced Grassmannian graphs of type $(k,n)$.

\begin{remark}
	Each refinement-equivalence class of Grassmannian trees (resp., forests) contains a unique contracted  Grassmannian tree (resp., forest).  This contracted tree (or forest) provides a canonical choice of representative for the equivalence class.  Note that a Grassmannian forest is contracted if and only if it has no adjacent white vertices and no adjacent black vertices. Similarly a plabic forest is contracted if and only if it is bipartite.
\end{remark}

\begin{remark}\label{remark:n-le-3}
	Every refinement-equivalence class of Grassmannian trees with three or fewer boundary vertices contains a single element. Given a single boundary vertex, one can construct precisely two Grassmannian trees. They consist of a single white or black boundary leaf. The only Grassmannian tree with two boundary vertices contains no internal vertices and one edge connecting the two boundary vertices. There are two Grassmannian trees with three boundary vertices, containing a single white or black trivalent vertex.
\end{remark}

\cref{fig:G-tree-9} depicts two refinement-equivalent Grassmannian trees with $9$ boundary vertices. 

\begin{figure}[h]
	\centering
	\begin{align*}
		\begin{gathered}
			\includegraphics[scale=0.4]{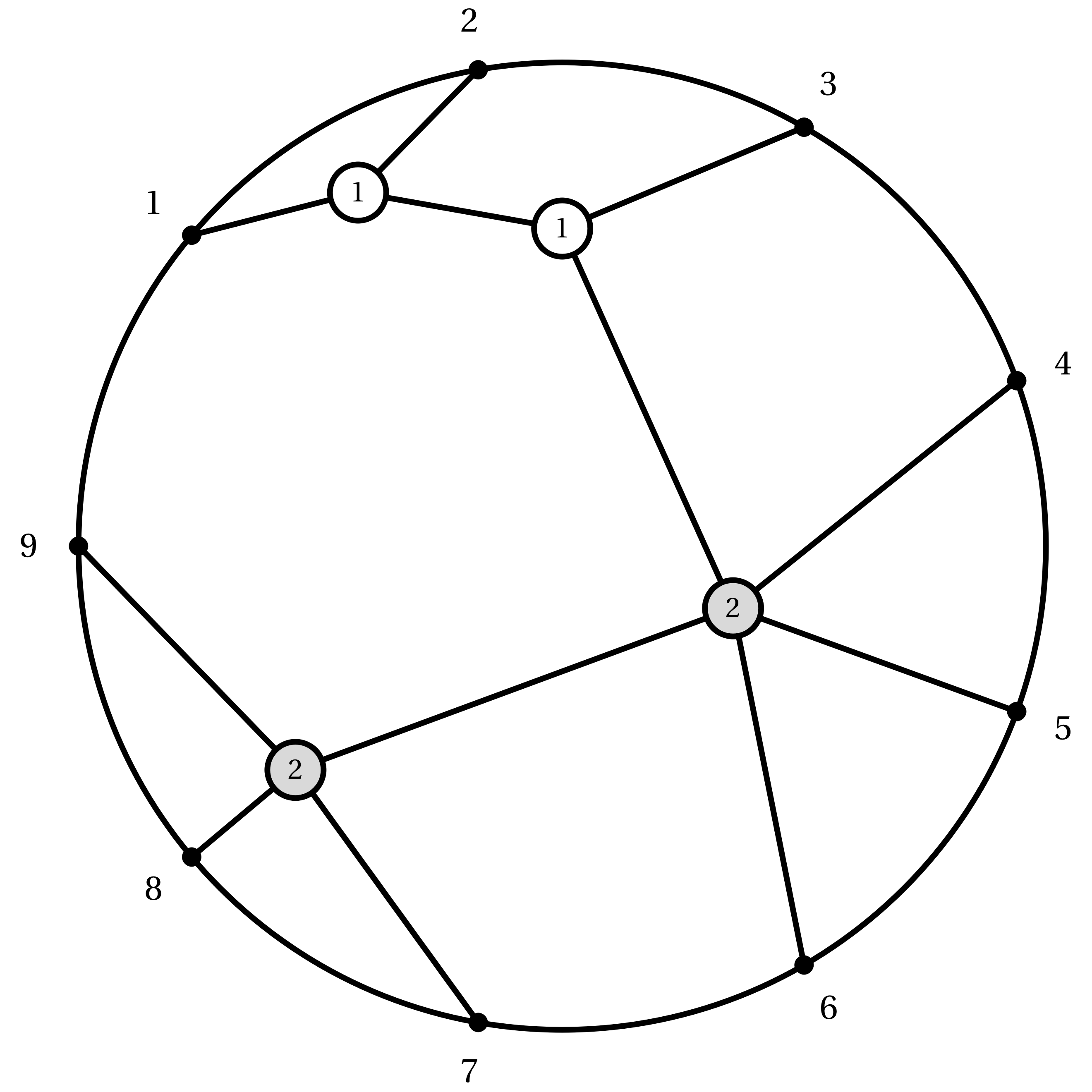}
		\end{gathered}
		\le
		\begin{gathered}
			\includegraphics[scale=0.4]{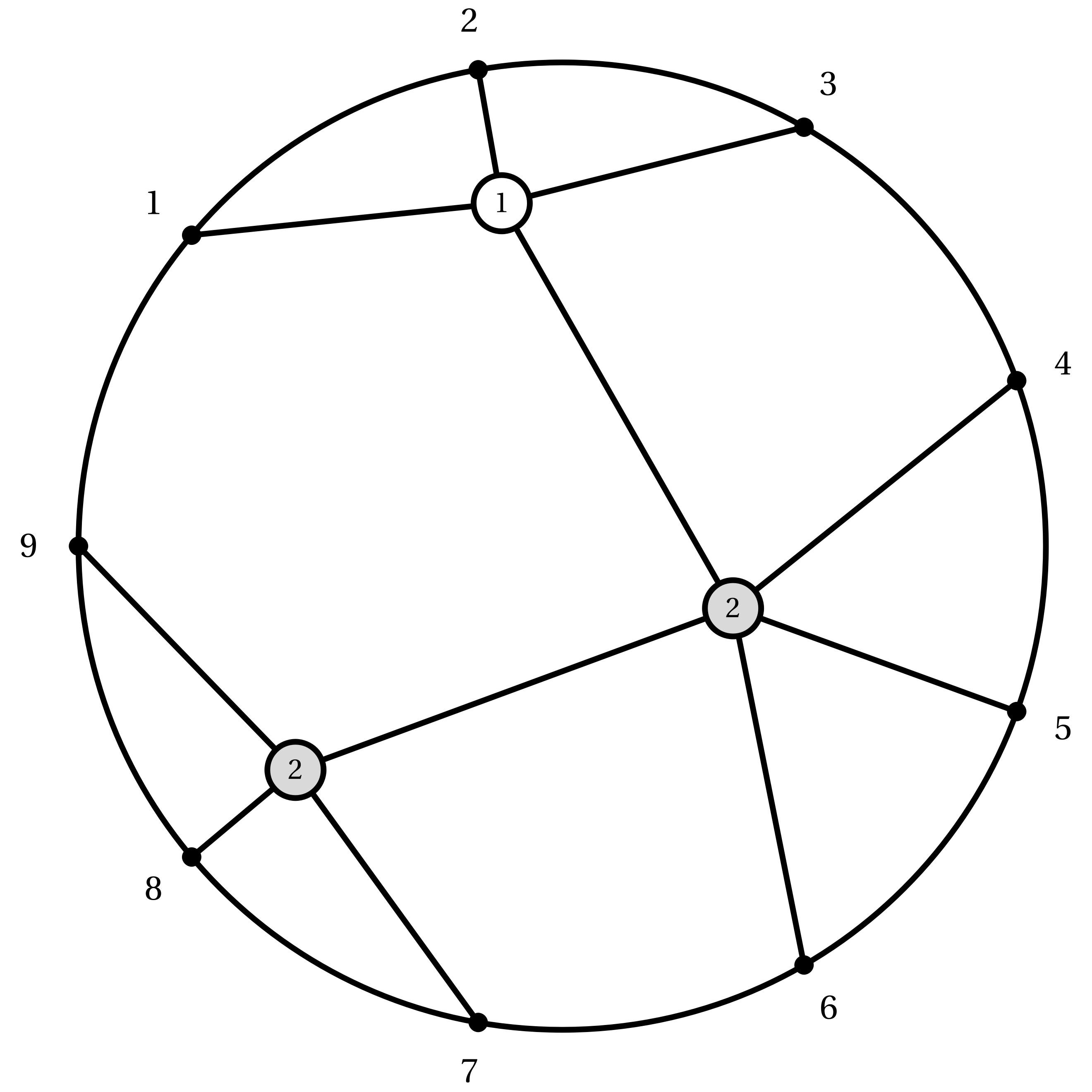}
		\end{gathered}
	\end{align*}
	\caption{Two refinement-equivalent Grassmannian trees with $9$ boundary vertices. The left tree is a refinement of the right one, and the right tree is contracted. The helicity of each vertex is printed inside the vertex.} 
	\label{fig:G-tree-9}
\end{figure}

We next define a natural dimension statistic associated to each refinement-equivalence class of Grassmannian forests; this will capture the dimension of the image of this positroid cell in the momentum amplituhedron.

\begin{definition}\label{def:mdim} 
	For every internal vertex $v$ in a Grassmannian forest, define  
	\begin{align}\label{eq:mom-dim-v}
		m(v) := \left\{\begin{array}{ll}
			2\deg(v)-4,& \text{$v$ is generic},\\
			\deg(v)-1,& \text{otherwise}.					
		\end{array}\right.
	\end{align}
	The \emph{momentum amplituhedron dimension} (or \emph{mom-dimension}) of a Grassmannian tree $T$ with $n$ boundary vertices, denoted by $\dim_\M(T)$, is defined as 
	\begin{align}\label{eq:mom-dim-T}
		\dim_\M(T):=\left\{\begin{array}{ll}
			n-1,& n\le 2,\\
			1+\sum_{v\in\Verts_\text{int}(T)}(m(v)-1),& \text{otherwise}.		
		\end{array}\right.
	\end{align}
	The \emph{momentum amplituhedron dimension} of a Grassmannian forest $F$, denoted by $\dim_\M(F)$, is the sum of the mom-dimensions of the Grassmannian trees in $F$:
	\begin{align}\label{eq:mom-dim-F}
		\dim_\M(F):=\sum_{T\in\Trees(F)}\dim_\M(T)\,.
	\end{align}
\end{definition}

\begin{remark}\label{remark:tree-map}
	Given a Grassmannian tree $T$ of type $(k,n)$, if we replace each internal vertex $v$ by a vertex of helicity $\deg(v)-h(v)$, we obtain a Grassmannian tree of type $(n-k,n)$. By \cref{def:mdim}, \eqref{eq:helicity-T} and \eqref{eq:multiplicity-T}, this map gives a dimension-preserving bijection between the Grassmannian trees of type $(k,n)$ and the Grassmannian trees of type $(n-k,n)$.  
\end{remark}

By \cite{Postnikov:2018jfq}, the helicity is an invariant of refinement-equivalence classes of Grassmannian graphs. The following gives an analogue of this statement for the mom-dimension.

\begin{lemma}\label{lem:invariant}
	The mom-dimension is an invariant of refinement-equivalence classes of Grassmannian forests.
\end{lemma}

\begin{proof}
	It is sufficient to prove that the mom-dimension for refinement-equivalent Grassmannian trees is the same. Without loss of generality, let $T$ and $T'$ be refinement-equivalent Grassmannian trees where $T$ is obtained from $T'$ by applying a single vertex uncontraction move to some non-generic internal vertex $v_\ast\in\Verts_\text{int}(T')$. Let $\Verts_\text{int}(T)\setminus\Verts_\text{int}(T')=\{v_1,v_2\}$ be the two non-generic internal vertices resulting from the uncontraction of $v_\ast$. Clearly $\deg(v_\ast)=\deg(v_1)+\deg(v_2)-2$. Then
	\begin{align*}
		1+(m(v_1)-1)+(m(v_2)-1)&= m(v_1)+m(v_2)-1 \\
		&=(\deg(v_1)-1)+(\deg(v_2)-1) -1 \\
		&=(\deg(v_1)+\deg(v_2)-2) - 1 \\
		&=\deg(v_\ast)-1=m(v_\ast)=1+(m(v_\ast)-1)\,,
	\end{align*}
	from which it follows that $\dim_\M(T)=\dim_\M(T')$.
\end{proof}

\begin{remark}
	For an internal vertex $v$ of type $(k,d)$ in a Grassmannian forest, we have that $m(v)$ is at most $k(d-k)$, the dimension of $\Gr_{k,d}^{\geq 0}$.  We have $m(v)=k(d-k)$ precisely when	$k=1,2,d-2,d-1$. Moreover, for a Grassmannian forest $F$ of type $(k,n)$, the mom-dimension of $F$ is at most the dimension of the positroid cell $S_F$ in $\Gr_{k,n}^{\ge 0}$, and these dimensions coincide when $k=0,1,2,n-2,n-1,n$.  
\end{remark}

As a warmup for proving our main result enumerating Grassmannian trees and forests (\cref{thm:GF-trees}), we first consider the simpler problem of enumerating contracted plabic trees and forests. Recall that a contracted plabic tree (respectively, forest) is simply a bipartite planar tree (respectively, forest).

\begin{theorem}\label{thm:GF-plabictrees}
	The number of contracted plabic trees (equivalently, bipartite planar trees) of type $(k,n)$ with mom-dimension $r$ is equal to the coefficient $[x^ny^kq^r]\mathcal{P}_\text{tree}(x,y,q)$, where 
	\begin{align}
		\label{eq:GF-plabictrees-1}\mathcal{P}_\text{tree}(x,y,q)&=x\left(1+y+yq\,C^{\lrangle{-1}}(x,y,q)\right),\text{ with }\\
		\label{eq:GF-plabictrees-2}C(x,y,q)&=\frac{x(1-q^2 x^2 y)}{(1+xq)(1+xyq)},
	\end{align}
	and the compositional inverse is with respect to the variable $x$.
	
	The number of contracted plabic forests of type $(k,n)$ with mom-dimension $r$ is given by $[x^ny^kq^r]\mathcal{P}_\text{forest}(x,y,q)$, where
	\begin{align}\label{eq:GF-plabicforest-1}
		x \mathcal{P}_\text{forest}(x,y,q)=\left(\frac{x}{1+\mathcal{P}_\text{tree}(x,y,q)}\right)^{\lrangle{-1}},
	\end{align}
	and the compositional inverse is with respect to the variable $x$. Equivalently, 
	\begin{align}\label{eq:GF-plabicforest-2}
		[x^n] \mathcal{P}_\text{forest}(x,y,q) = \frac{1}{n+1} [x^n] 
		\left( 1+\mathcal{P}_\text{tree}(x,y,q) \right)^{n+1}.
	\end{align}
\end{theorem}

\begin{remark} 
	The numbers of contracted plabic trees of type $(k,n)$ refine the \emph{large Schroeder numbers}, and appear as \cite[A175124]{OEIS}.  By \cite[Section 12]{Parisi:2021oql}, these are also equinumerous with the separable permutations on $[n-1]$ with $k-1$ descents. $C^{\lrangle{-1}}(x,y,q)$ coincides with the generating function given in \cite[Example 1.6.7]{drake2008inversion} with $(\alpha,\beta)=(q,yq)$ or $(yq,q)$. The polynomial enumerating separable permutations according to their descents was also studied in \cite{fu2018two}.
\end{remark}

\begin{remark}	
	\cref{thm:GF-plabictrees} implies that $\mathcal{P}_\text{tree}(x,y,q)$ and $\mathcal{P}_\text{forest}(x,y,q)$ are algebraic generating functions of degree $3$ and $4$, respectively, satisfy the following relations:
	\begin{align}
		\begin{autobreak}
			0 
			= 
			\mathcal{P}_\text{tree}^3+ x (y (q x-3)-3)\mathcal{P}_\text{tree}^2
			+ x^2 (3-y (q x (y+1)-3 y-5))\mathcal{P}_\text{tree}
			-x^3 (1+y (y (-q x+y+2)+2))\,,
		\end{autobreak}\\
		\begin{autobreak}
			0
			=
			q (x-1) x^2 y (x y-1)\mathcal{P}_\text{forest}^4
			+ (x (x (-x (y+1) (-q y+y^2+y+1)+y (-2 q+3 y+5)+3)-3 (y+1))+1)\mathcal{P}_\text{forest}^3
			+ (x^2 (y (q-3 y-5)-3)+6 x (y+1)-3)\mathcal{P}_\text{forest}^2
			-3 (x y+x-1)\mathcal{P}_\text{forest}
			-1\,.
		\end{autobreak}		
	\end{align}
\end{remark}

\begin{remark}
	We may as well set $q$ to $1$ in \eqref{eq:GF-plabictrees-1} and \eqref{eq:GF-plabictrees-2} because $\dim_{\mathcal{M}}(T)=n-1$ for any contracted plabic tree $T$ on $n$-leaves and so $[x^ny^kq^r]\mathcal{P}_\text{tree}(x,y,q) = \delta_{r,n-1}[x^ny^k]\mathcal{P}_\text{tree}(x,y,1)$.
\end{remark}

Special care is required to enumerate the \emph{contracted} plabic trees, as defined in \cref{def:contracted}. In particular, we need to ensure that the set of all subtrees with a fixed set of boundary vertices and with all internal vertices white, count as a single contribution; similarly for the set of all subtrees with a fixed set of boundary vertices and all internal vertices black.  Note that any tree with $n$ boundary vertices and all internal vertices white (respectively, black) has helicity $1$ (respectively, $n-1$). To this end, we define a statistic $e:\Z_{\ge 3}\to \R$ on internal vertices of a tree which depends only on their degree. It is defined so that the sum over all trees (with a fixed number of boundary vertices) weighted by $e(\deg(v))$ for each internal vertex $v$ equals 1.

\begin{lemma}\label{thm:equiv-c}
	Let $e:\Z_{\ge 3}\to \R$ be the function where $e(n)=(-1)^{n-1}$ for each $n\in\Z_{\ge3}$. Then the function $h:\Z_{\ge 3}\to\R$ given by
	\begin{align}\label{eq:equiv-g}
		h(n)=\sum_{T\in \mathcal{T}_n}\prod_{v\in\Verts_\text{int}(T)}e(\deg(v))\,,
	\end{align}
	has the property that $h(n)=1$.
\end{lemma}

\begin{proof}[Proof of \cref{thm:equiv-c}] 
	Notice that \eqref{eq:equiv-g} in \cref{thm:equiv-c} is precisely \eqref{eq:tree-h} in \cref{thm:tree} with $f=e$. If we let $H(x)=x^2+\sum_{n\ge 3}h(n)x^n$ and $F(x)=\sum_{n\ge3}e(n)x^n=\frac{x^3}{1+x}$, then by \cref{thm:tree} we have that 
	\begin{align*}
		\frac{1}{x}H(x)=\left(x-\frac{1}{x}F(x)\right)^{\lrangle{-1}} = \left(\frac{x}{1+x}\right)^{\lrangle{-1}}=\frac{x}{1-x}\,,
	\end{align*} 
	which means that $H(x) = \frac{x^2}{1-x} = x^2+\sum_{n\ge 3}x^n$ and $h(n)=1$.
\end{proof}

\begin{proof}[Proof of \cref{thm:GF-trees}] 
	To simplify notation, we will suppress any functional dependence on the variables $y$ and $q$, and write $\mathcal{P}_\text{tree}(x)$ for $\mathcal{P}_\text{tree}(x,y,q)$.
	
	Given a plabic tree $T$ of type $(k,n)$ with $n\ge3$ and mom-dimension $r$, we can express $x^ny^kq^r$ using \eqref{eq:helicity-T} and \eqref{eq:mom-dim-T} as 
	\begin{align}\label{eq:GF-plabictrees-contrib}
		x^ny^kq^r=x^nyq\prod_{v\in\Verts_\text{int}(T)}y^{h(v)-1}q^{m(v)-1}\,,
	\end{align}
	where $q^{m(v)-1}=q^{\deg(v)-2}$ for internal vertices. Comparing the right hand side of the equality in \eqref{eq:GF-plabictrees-contrib} with \eqref{eq:tree-h} motivates the following definition of the function $f$, whose value $f(d)$ encodes the two ways of coloring an internal vertex of degree $d$ (while keeping track of both $h(v)$ and $m(v)$).
	
	We define $f:\Z_{\ge 3}\to\QQ(y,q)$ by
	\begin{align*}
		f(d)=q^{d-2} e(d)(1+y^{d-2})\,,
	\end{align*}
	where $e(d)=(-1)^{d-1}$ is the statistic introduced in \cref{thm:equiv-c}. Note that the prefactor $e(d)$ is included so
	that all refinement-equivalent plabic (sub-)trees with only white vertices (or with only black vertices), are counted as a single contribution. 
	
	Let $h:\Z_{\ge 3}\to\QQ(y,q)$ be defined in terms of $f$ as in \cref{thm:tree}. Let $F(x)=\sum_{d\ge3}f(d)x^d$ and $H(x)=x^2+\sum_{n\ge3}h(n)x^n$. Then we can concretely compute 
	\begin{equation}
		F(x)=x^3 q\left(\frac{1}{1+xq}+\frac{y}{1+xyq}\right),
	\end{equation} and 
	$H(x)$ may be computed in terms of $F(x)$ using \cref{thm:tree}. Finally, $\mathcal{P}_\text{tree}(x)$ is given by
	\begin{align*}
		\mathcal{P}_\text{tree}(x)=x(1+y)+yq H(x)\,.
	\end{align*}
	We add the term $x(1+y)$ so that $[x^1]\mathcal{P}_\text{tree}(x)=(1+y)$ accounts for the two plabic trees with a single boundary vertex. Setting $C(x)=x-\frac{1}{x}F(x)$, the result for $\mathcal{P}_\text{tree}(x,y,q)$ follows.
	
	The first statement about $\mathcal{P}_\text{forest}(x,y,q)$ now follows from \cref{thm:forest}. The second statement \eqref{eq:GF-plabicforest-2} follows from \eqref{eq:GF-plabicforest-1} by applying Lagrange Inversion (\cref{thm:LIF}). Explicitly, if we write $R(x) = \frac{x}{1+\mathcal{G}_\text{tree}(x,y,q)}$, then 
	\begin{align*}
		(n+1) [x^{n+1}] R^{\lrangle{-1}}(x) = [x^{n}] \left(\frac{x}{R(x)}\right)^{n+1} = [x^{n}] (1+\mathcal{P}_\text{tree}(x,y,q))^{n+1}.
	\end{align*}
\end{proof}

We now return to the case of contracted Grassmannian trees and forests.
\begin{theorem}\label{thm:GF-trees}
	The number of contracted Grassmannian trees of type $(k,n)$ with mom-dimension $r$ is given by $[x^ny^kq^r]\mathcal{G}_\text{tree}(x,y,q)$ where 
	\begin{align}
		\label{eq:GF-trees-1}\mathcal{G}_\text{tree}(x,y,q)&=x\left(1+y+yq\,C^{\lrangle{-1}}(x,y,q)\right),\text{ with }\\
		\label{eq:GF-trees-2}C(x,y,q)&=\frac{x(1-x (1+y) q^2-x^2 y q^2 (1+q-q^2)-x^4 y^2 q^5 (1+q))}{(1+xq)(1+xyq)(1-xq^2)(1-xyq^2)},
	\end{align}
	and the compositional inverse is with respect to the variable $x$.
	
	The number of contracted Grassmannian forests of type $(k,n)$ with mom-dimension $r$ is given by $[x^ny^kq^r]\mathcal{G}_\text{forest}(x,y,q)$ where
	\begin{align}\label{eq:GF-forest-1}
		x \mathcal{G}_\text{forest}(x,y,q)=\left(\frac{x}{1+\mathcal{G}_\text{tree}(x,y,q)}\right)^{\lrangle{-1}},
	\end{align}
	and the compositional inverse is  with respect to the variable $x$. Equivalently, 
	\begin{align}\label{eq:GF-forest-2}
		[x^n] \mathcal{G}_\text{forest}(x,y,q) = \frac{1}{n+1} [x^n] 
		\left( 1+\mathcal{G}_\text{tree}(x,y,q) \right)^{n+1}.
	\end{align}
\end{theorem}

Our proof is quite analogous to the previous one.

\begin{proof}[Proof of \cref{thm:GF-trees}] 	
	As before, given a Grassmannian tree $T$ of type $(k,n)$ with $n\ge3$ and mom-dimension $r$, we can express $x^ny^kq^r$ using \eqref{eq:helicity-T} and \eqref{eq:mom-dim-T} as 
	\begin{align}\label{eq:GF-trees-contrib}
		x^ny^kq^r=x^nyq\prod_{v\in\Verts_\text{int}(T)}y^{h(v)-1}q^{m(v)-1}\,,
	\end{align}
	where $q^{m(v)-1}=q^{2\deg(v)-5}$ for generic vertices, while $q^{m(v)-1}=q^{\deg(v)-2}$ for non-generic vertices. Comparing the right hand side of the equality in \eqref{eq:GF-trees-contrib} with \eqref{eq:tree-h} motivates the following definition of the function $f$, whose value $f(d)$ encodes the different ways to decorate an internal vertex of degree $d$ (while keeping track of both $h(v)$ and $m(v)$).
	
	We define $f:\Z_{\ge 3}\to\QQ(y,q)$ by
	\begin{align*}
		f(d)=q^{2d-5}\sum_{k=2}^{d-2}y^{k-1}+q^{d-2} e(d)(1+y^{d-2})\,,
	\end{align*}
	where $e(d)=(-1)^{d-1}$. Note that the expression $q^{2d-5} \sum_{k=2}^{d-2}y^{k-1}$ in $f(d)$ enumerates the $d-3$ choices of helicities for a generic internal degree $d$ vertex $v$, while the expression $q^{d-2} e(d)(1+y^{d-2})$ enumerates the remaining two non-generic choices. The prefactor $e(d)$ is included so that all refinement-equivalent Grassmannian trees with only white vertices (or with only black vertices), are counted as a single contribution. 
	
	Let $h:\Z_{\ge 3}\to\QQ(y,q)$ be defined in terms of $f$ as in \cref{thm:tree}. Let $F(x)=\sum_{d\ge3}f(d)x^d$ and $H(x)=x^2+\sum_{n\ge3}h(n)x^n$. Then we can concretely compute 
	\begin{equation}
		F(x)=x^3 q\left(\frac{1}{1+xq}+\frac{y}{1+xyq}+\frac{xyq^2}{(1-x q^2)(1-xyq^2)}\right),
	\end{equation} and 
	$H(x)$ may be computed in terms of $F(x)$ using \cref{thm:tree}. Finally, $\mathcal{G}_\text{tree}(x)$ is given by $\mathcal{G}_\text{tree}(x)=x(1+y)+yq H(x)$. Note that we added the term $x(1+y)$ to account for the two Grassmannian trees with a single boundary vertex. Setting $C(x)=x-\frac{1}{x}F(x)$, the result for $\mathcal{G}_\text{tree}(x,y,q)$ follows. The proofs of the statements about $\mathcal{G}_\text{forest}(x,y,q)$ are exactly the same as in the previous proof.
\end{proof}

See \cref{tbl:G-forest} for examples of $[x^ny^k]\mathcal{G}_\text{forest}(x,y,q)$ for $n=4$ through $12$ and for  $2\le k\le \lfloor\frac{n}{2}\rfloor$.

\begin{remark}	
	It follows from \cref{thm:GF-trees} that $\mathcal{G}_\text{tree}(x,y,q)$ (resp., $\mathcal{G}_\text{forest}(x,y,q)$) is an algebraic generating function of degree $5$ (resp., $6$) satisfying \eqref{eq:GF-tree-rel} (resp., \eqref{eq:GF-forest-rel}).
\end{remark}

%%%%%%%%%%%%%%%%%%%%%%%

\section{Separable permutations and Grassmannian tree permutations}\label{sec:perm} 

In this section we explain how the results of the previous section are related to the enumeration of permutations. We start by defining the decorated trip permutation associated to a reduced Grassmannian graph \cite[Definition 4.5]{Postnikov:2018jfq}. Decorated trip permutations are in bijection with refinement-equivalence classes of reduced Grassmannian graphs. We also note that the Grassmannian graphs that we are interested in in this paper (Grassmannian forests) are automatically reduced.

\begin{definition}[{\cite[Definition 4.5]{Postnikov:2018jfq}}]\label{def:rules}
	A one-way \emph{trip} $\alpha$ in a Grassmannian graph $G$ is a directed walk along edges of $G$ that starts and ends at some boundary vertices, satisfying the following \emph{rules-of-the-road}: For each internal vertex $v\in \Verts_\text{int}(G)$ with adjacent edges labelled $a_1,\ldots,a_d$ in the clockwise order, where $d=\deg(v)$, if $\alpha$ enters $v$ through the edge $a_i$, it leaves $v$ through the edge $a_{j}$, where $j=i+h(v) \pmod d$.	
	
	A \emph{decorated permutation} on $n$ letters is a permutation $w:[n]\to [n]$ in which fixed points are coloured either black or white (and consequently	denoted $w(i)=\underline{i}$ and $w(i)=\overline{i}$).
	
	The decorated permutation $w_G$ of a reduced Grassmannian graph  is defined as follows:
	\begin{enumerate}
		\item If the trip starting at the boundary vertex $b_i$ ends at the boundary vertex $b_j$ for $j\neq i$, then $w_G(i) = j$. 
		\item If the trip starting at boundary vertex $b_i$ ends at $b_i$, then either $w_G(i) = \underline{i}$ or $w_G(i) = \overline{i}$, based on whether the leaf $v$ incident to $b_i$ has $h(v)=0$ or $h(v)=1$.
	\end{enumerate}
	We define an \emph{antiexcedance} of a decorated permutation $w:[n]\to[n]$ is an element $i\in [n]$ such that either $w^{-1}(i)>i$ or $w(i) = \overline{i}$.  
\end{definition}

The helicity $h(G)$ of a reduced Grassmannian graph is equal to the number of {antiexcedances} of the decorated trip permutation $w_G$ \cite{Postnikov:2018jfq}. Therefore we will also denote the number of antiexcedances of this permutation as $h(w_G):=h(G)$.

As an example, the decorated trip permutation associated to both graphs in \cref{fig:G-tree-9} is $(2, 3, 5, 6, 8, 1, 9, 4, 7)$, which has three antiexcedances; this corresponds to the fact that the graphs have helicity $3$.

Since each refinement-equivalence class of Grassmannian forests has a uniquely associated decorated trip permutation, \cref{lem:invariant} allows us to define the mom-dimension of each decorated permutation $w_F$ associated to a Grassmannian forest $F$, that is, we define $\dim_\M(w_F):=\dim_\M(F)$.

Let us first interpret \cref{thm:GF-plabictrees} in terms of permutations. One option is to interpret \cref{thm:GF-plabictrees}  as counting the trip permutations of plabic trees, enumerating them according to $n$, number of antiexcedances $k$, and mom-dimension. Another option is to use the results of \cite[Section 12]{Parisi:2021oql} that the contracted plabic trees of type $(k,n)$ are in bijection with the \emph{separable permutations} on $n-1$ letters with $k-1$ descents, where a separable permutation can be defined as a permutation which avoids the patterns $2413$ and $3142$ \cite{bose1998pattern,kitaev2011patterns}. Therefore if we specialize $q=1$ in \cref{thm:GF-plabictrees}, we have the following.
\begin{corollary}
	The number of separable permutations on $n-1$ letters with $k-1$ descents is given by 
	$$[x^n y^k]  
	\mathcal{P}_\text{tree}(x,y,q=1) =
	[x^n y^k] 
	\left(x+xy+xy\,
	{\left(\frac{x(1- x^2 y)}{(1+x)(1+xy)}\right)}^{\lrangle{-1}} \right) .$$
\end{corollary}
We note that the generating function above is essentially the same one that appears in \cite[Example 1.6.7]{drake2008inversion}. 

One can also define separable permutations as the permutations that can be built by applying \emph{direct sums} and \emph{skew sums}, starting	from the trivial permutation $1$, where the direct sum operation is defined as follows.
\begin{definition}
	The \emph{direct sum} of two permutations $\sigma$ on $n_\sigma$ letters and $\tau$ on $n_\tau$ letters is a permutation $\sigma\oplus\tau$ on $n_\sigma+n_\tau$ letters defined as follows: 
	\begin{align}
		(\sigma\oplus\tau)(i)\coloneqq\left\{
		\begin{array}{ll}
			\sigma(i) & \text{if $1\le i\le n_\sigma$,}\\
			\tau(i-n_\tau)+n_\tau & \text{if $n_\sigma+1\le i\le n_\sigma+n_\tau$.}
		\end{array}
		\right.
	\end{align}
	
	This operation is illustrated in \cref{fig:permutation-operations} (left).
\end{definition}
\begin{figure}[h]
	\centering
	\begin{minipage}{.5\textwidth}
		\centering
		\includegraphics[scale=0.4]{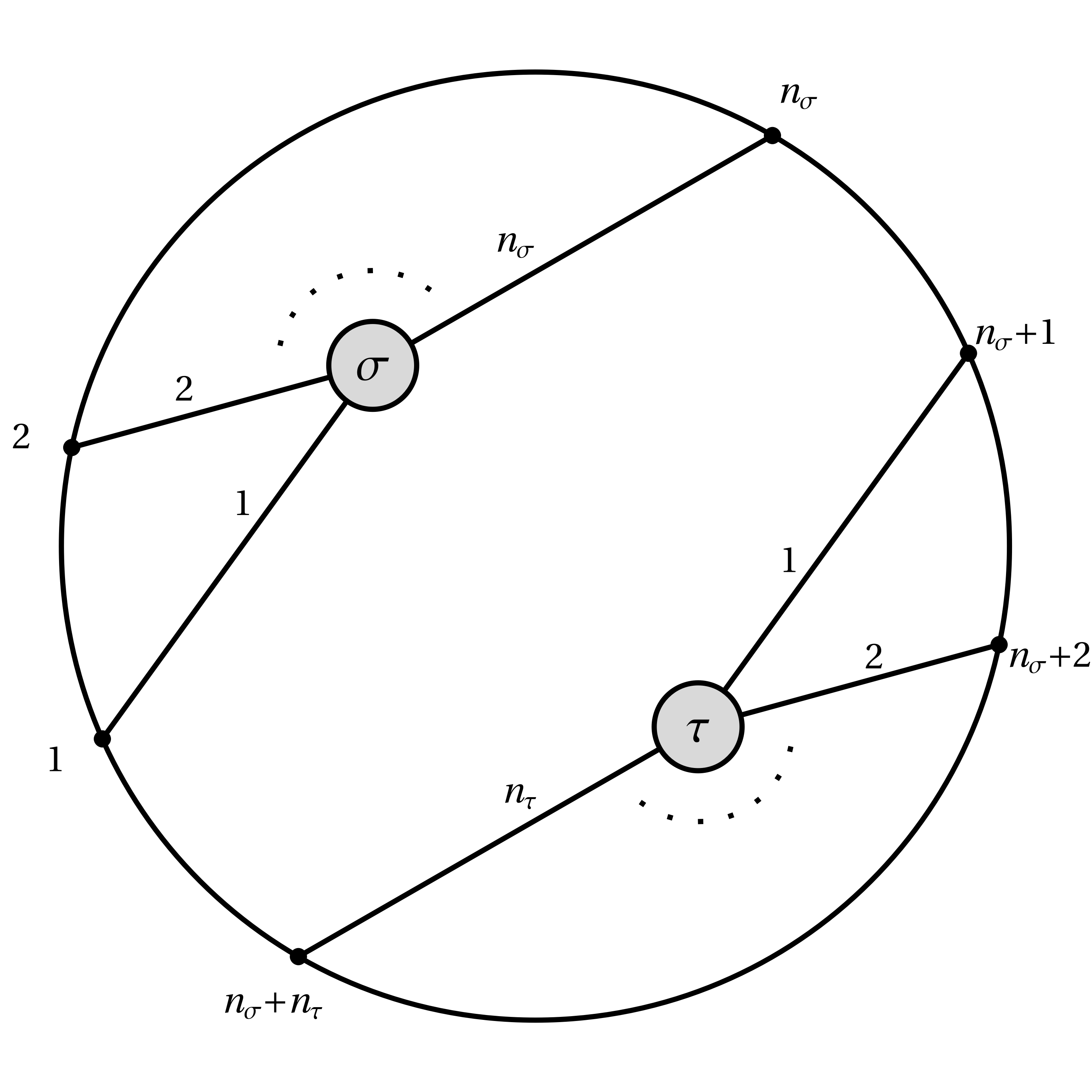}
	\end{minipage}%
	\begin{minipage}{.5\textwidth}
		\centering
		\includegraphics[scale=0.4]{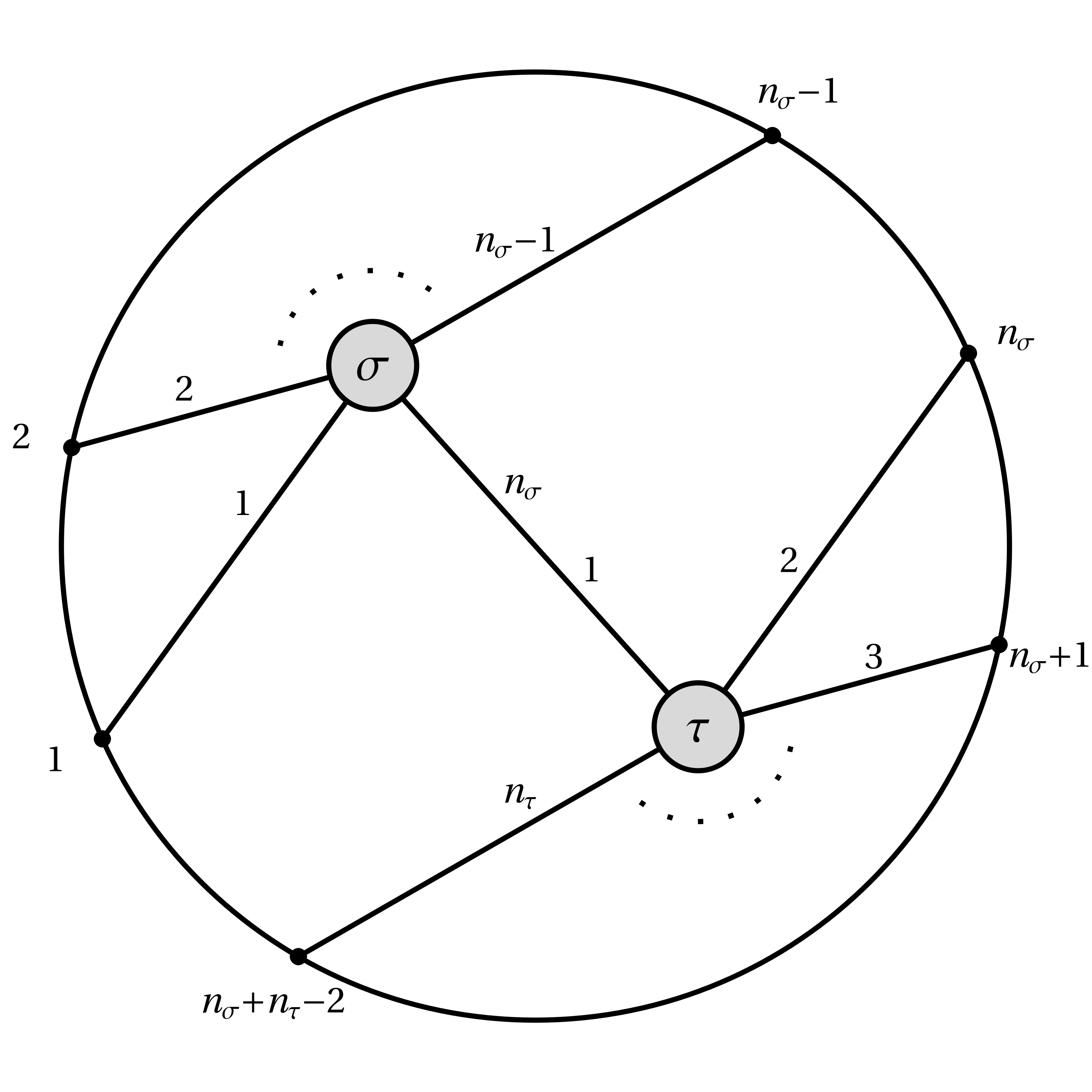}
	\end{minipage}
	\caption{Direct sum (left) and amalgamation (right) of two permutations $\sigma$ on $n_\sigma$ letters and $\tau$ on $n_\tau$ letters.}
	\label{fig:permutation-operations}
\end{figure}

Let us now interpret \cref{thm:GF-trees} in terms of permutations. We first need to describe the decorated permutations that arise from contracted Grassmannian trees and graphs.

For $n\ge2$ and $1\le k\le n-1$, let
\begin{align*}
	\pi_{k,n}\coloneqq (k+1,\ldots,n,1,2,\ldots,k)=\begin{pmatrix}
		1 & 2 & \cdots & n-k & n-k+1 & \cdots & n\\
		\downarrow & \downarrow & & \downarrow & \downarrow & & \downarrow\\
		k+1 & k+2 & \cdots & n & 1 & \cdots & k 
	\end{pmatrix},
\end{align*}
and for $n=1$, let $\pi_{0,1}\coloneqq(\underline{1})$ and $\pi_{1,1}\coloneqq(\overline{1})$ be the two decorated permutations on one letter decorated black and white, respectively.

\begin{definition}
	Given permutations $\sigma$ on $n_\sigma$ letters and $\tau$ on $n_\tau $ letters, where $n_\sigma,n_\tau\ge2$, the \emph{amalgamation} $\sigma\star\tau$ is a permutation on $n_\sigma+n_\tau-2$ letters defined as follows:
	\begin{align}
		(\sigma\star\tau)(i)\coloneqq\left\{
		\begin{array}{ll}
			\sigma(i) & \text{if $1\le i\le n_\sigma-1$ and $\sigma(i)\ne n_\sigma$,}\\
			\tau(1) & \text{if $1\le i\le n_\sigma-1$ and $\sigma(i)=n_\sigma$,}\\
			\tau(i-n_\sigma+2) & \text{if $n_\sigma\le i \le n_\sigma+n_\tau-2$ and $\tau(i-n_\sigma+2)\ne 1$,}\\
			\sigma(n_\sigma) & \text{if $n_\sigma\le i \le n_\sigma+n_\tau-2$ and $\tau(i-n_\sigma+2)= 1$.}
		\end{array}
		\right.
	\end{align}
	This operation is illustrated in \cref{fig:permutation-operations} (right).
\end{definition}

\begin{definition}
	The \emph{cyclic rotation} of a permutation $\pi=(\pi_1,\ldots,\pi_{n-1},\pi_n)$ is the permutation $\text{cyc}(\pi)$ defined by $(\text{cyc}(\pi))(i)\coloneqq\pi(i-1)+1$, i.e.\ $\text{cyc}(\pi)=(\pi_n+1,\pi_1+1,\ldots,\pi_{n-1}+1)$ (modulo $n$). 
\end{definition}
The above operation corresponds to taking a decorated permutation arising from a Grassmannian graph on vertices $1,2,\dots,n$, then adding $1$ (modulo $n$) to each boundary vertex and computing the new decorated permutation.

\begin{definition}
	A \emph{Grassmannian tree permutation} is a permutation built by amalgamating permutations of the form $\pi_{k,n}$ (for $n \geq 3$) and possibly applying cyclic rotations. We also consider the permutations $(1)$ and $(2,1)$ on one and two letters, respectively, to be Grassmannian tree permutations.
\end{definition}

The following statement is easy to verify from the definitions.
\begin{proposition}
	Grassmannian tree permutations are precisely the decorated permutations obtained from Grassmannian trees by applying the map from \cref{def:rules}.
\end{proposition}

Now we obtain from \cref{thm:GF-trees} the following corollary.
\begin{corollary} 
	The number of Grassmannian tree permutations on $n$ letters with $k$ antiexcedances and momentum amplituhedron $r$ is given by $[x^ny^kq^r]\mathcal{G}_\text{tree}(x,y,q)$ where $\mathcal{G}_\text{tree}(x,y,q)$  is as in \cref{thm:GF-trees}.
\end{corollary} 

One can also interpret \cref{thm:GF-trees} as enumerating ``Grassmannian forest permutations,'' which are built from Grassmannian tree permutations by using the direct sum and cyclic rotation operations.

%%%%%%%%%%%%%%%%%%%%%%%

\section{The Momentum Amplituhedron} \label{sec:momentum}

Remarkably, the totally nonnegative Grassmannian underpins the structure of scattering amplitudes in planar $\mathcal{N}=4$ supersymmetric Yang-Mills (SYM) theory \cite{Arkani-Hamed:2012zlh}. In 2013, the (tree) amplituhedron $\mathcal{A}_{n,k,m}$ was introduced as a geometric object which encodes tree-level scattering amplitudes for $\mathcal{N}=4$ SYM in momentum twistor space \cite{Arkani-Hamed:2013jha}. It is defined as the image of the totally nonnegative Grassmannian under a linear map induced by a positive matrix. Then in 2019, the momentum amplituhedron $\mathcal{M}_{n,k}=\mathcal{M}_{n,k,4}$ was discovered as an analogue of the amplituhedron but defined in momentum space \cite{Damgaard:2019ztj}. Like the amplituhedron, the momentum amplituhedron is defined as the image of the totally nonnegative Grassmannian under a particular map. Its boundary stratification was extensively studied in \cite{Ferro:2020lgp} and a (conjectural) description for its boundaries was given in terms of Grassmannian forests, although these objects were not known to the authors then. Moreover, each instance of the momentum amplituhedron computed in \cite{Ferro:2020lgp} was shown to have Euler characteristic $1$. In this section, we use the results of the previous section to construct the full generating function for the momentum amplituhedron and we prove that its Euler characteristic is $1$. 

\subsection{The momentum amplituhedron and its boundary stratification}

\begin{definition}
	Define the \emph{twisted nonnegative part} of $\Gr_{k,n}$ to be:
	\begin{align}
		\Gr^{\ge 0,\tau}_{k,n} = \lbrace V \in \Gr_{k,n}:
		(-1)^{\mbox{inv}(I, [n] \setminus I)}\Delta_{[n] \setminus I}(V) \geq 0  \text{ for all }
		I \in {[n] \choose n-k} \rbrace\,,
	\end{align}
	where $\mbox{inv}(A,B)=\# \lbrace a \in A, b \in B | a > b \rbrace $ denotes the inversion number.
\end{definition}

One can verify as in \cite[Lemma 1.11]{Karp:2015duv} that if $\{\Delta_I(V)\}$ are the Pl\"ucker coordinates of a point $V$ in $\Gr_{k,n}$, then the orthogonal complement $V^{\perp} \in \Gr_{n-k,n}$ is represented by a point with Pl\"ucker coordinates $\Delta_{J}(V^{\perp}) = (-1)^{\mbox{inv}(J, [n] \setminus J)}\Delta_{[n] \setminus J}(V)$ for $J\in {[n] \choose n-k}$.

\begin{definition}
	For $n,p$ with $n \geq p$, define $\mbox{Mat}^{>0}_{n,p}$ to be the set of real $n \times p$ matrices whose maximal minors (Pl\"ucker coordinates) are all positive and its \emph{twisted totally positive part} as
	\begin{equation}
		\mbox{Mat}^{>0,\tau}_{n,p} = \lbrace A \in \mbox{Mat}_{n,p}: (-1)^{\mbox{inv}(I, [n] \setminus I)}\Delta_{[n] \setminus I}(A) > 0 
		\text{ for all }I \in {[n] \choose n-p} \rbrace\,.
	\end{equation}
\end{definition}

\begin{definition}
	A subset $I$ of $[n]$ is said to be \emph{cyclically consecutive} if its elements, or the elements of its complement in $[n]$, are consecutive.
\end{definition}

\begin{definition}[{\cite[Section 2.2]{Damgaard:2019ztj}}]\label{def:momamp}
	Let $k,n$ be integers with $2\le k\le n-2$. Let $\tilde{\Lambda} \in \mbox{Mat}^{>0}_{n, k+2}$ and $\Lambda \in \mbox{Mat}^{>0,\tau}_{n, n-k+2}$ be a pair of matrices, called the \emph{kinematic data}. Given an element $[C]\in\Gr_{k,n}^{\ge 0}$, with orthogonal complement $[C^\perp] \in \Gr_{n-k,n}$, let $\tY = C\tilde{\Lambda}$ and $Y = C^\perp\Lambda$. The \emph{momentum amplituhedron map} $\Phi_{\tilde{\Lambda},\Lambda}: \Gr^{\ge 0}_{k,n} \rightarrow \Gr_{k,k+2} \times \Gr_{n-k,n-k+2}$ is defined by $\Phi_{\tilde{\Lambda},\Lambda}(C)\coloneqq (\tY,Y)$ and the \emph{momentum amplituhedron} 	$\mathcal{M}_{n,k}(\tilde{\Lambda},\Lambda)$ is the image of $\Gr_{k,n}^{\geq 0}$ under $\Phi_{\tilde{\Lambda},\Lambda}$.
	
	For $i,j\in[n]$, we define $[\tY ij]=\det(\tY_1,\ldots,\tY_k,\tilde{\Lambda}_{i},\tilde{\Lambda}_{j})$ to be the determinant of the $(k+2)\times(k+2)$ matrix with rows  $\tY_1,\ldots,\tY_k,\tilde{\Lambda}_{i},\tilde{\Lambda}_{j}$, and we define $\langle Y ij\rangle=\det(Y_1,\ldots,Y_{n-k},\Lambda_{i},\Lambda_{j})$ to be the determinant of the $(n-k+2)\times(n-k+2)$ matrix with rows  $Y_1,\ldots,Y_{n-k},\Lambda_{i},\Lambda_{j}$. We require that the kinematic data satisfy $\sum_{i<j\in I}[\tY ij]\langle Y ij\rangle > 0$ on $\Gr_{k,n}^{> 0}$ for every cyclically consecutive subset $I$ of $[n]$ with $|I| \geq 2$; this is a necessary condition for the codimension-one boundaries of the momentum amplituhedron to correspond to factorization channels of the scattering amplitude \cite[Section 2.3]{Damgaard:2019ztj}.
\end{definition}

While $\Gr_{k,k+2} \times \Gr_{n-k,n-k+2}$ has dimension $2n$, $\M_{n,k}$  has dimension $2n-4$.

\begin{remark} \label{remark:independent-of-kinematic-data}
	The combinatorial properties of $\mathcal{M}_{n,k}(\tilde{\Lambda},\Lambda)$ are conjectured to be independent of the choice of $(\tilde{\Lambda},\Lambda)$. Consequently, we will omit our choice and simply write $\mathcal{M}_{n,k}$.   
\end{remark}

\begin{definition}
	Given a positroid cell $S_{\sigma}$ (resp., $S_G$) of $\Gr_{k,n}^{\ge 0}$, we write $\Phi_{\sigma}^\circ = \Phi(S_\sigma)$ (resp., $\Phi_{G}^\circ = \Phi(S_G)$) and $\Phi_{\sigma} = \overline{\Phi(S_\sigma)}$ (resp., $\Phi_{G} = \overline{\Phi(S_G)}$), omitting our choice of kinematic data following \cref{remark:independent-of-kinematic-data}. We refer to $\Phi_{\sigma}^\circ$ (resp., $\Phi_{G}^\circ$) as a \emph{stratum} of $\mathcal{M}_{n,k}$ and we denote its \emph{dimension} by $\dim \Phi_{\sigma}^\circ$ (resp., $\dim \Phi_{G}^\circ$).
\end{definition}

The following definition first appeared in \cite[Section 2.4]{Lukowski:2019kqi}.

\begin{definition}
	Let $S_{\sigma_{k,n}}$ denote the unique top-dimensional positroid cell of $\Grk$ which is the interior of $\Grk$. Given a positroid cell $S_{\sigma}$ of $\Grk$, we say that $\Phi^{\circ}_{\sigma}$ is a \emph{boundary stratum} (or simply \emph{boundary}) of 
	$\mathcal{M}_{n,k}$ if 
	\begin{itemize}
		\item $\Phi^{\circ}_{\sigma} \cap \Phi^{\circ}_{\sigma_{k,n}} = \emptyset$, and
		\item for any positroid cell $S_{\sigma'}$ whose closure contains $S_{\sigma}$, we have $\dim \Phi^{\circ}_{\sigma'} > \dim \Phi^{\circ}_{\sigma}$.
	\end{itemize}
\end{definition}

The boundary stratification of the momentum amplituhedron was extensively studied in \cite{Ferro:2020lgp} using \texttt{amplituhedronBoundaries.m} \cite{Lukowski:2020bya}, a Mathematica package. The package employs a recursive routine, initially developed for the $m=2$ amplituhedron in \cite{Lukowski:2019kqi}, which exploits the positroid stratification\footnote{The package \texttt{positroids.m} \cite{Bourjaily:2012gy} implements the positroid stratification in Mathematica.} of $\Grk$. The algorithm in \cite{Lukowski:2019kqi} (conjecturally) generates all boundaries as per the above definition.

Let $\mathcal{S}_{n,k}$ denote the set of decorated permutations $\sigma$ such that $\Phi^{\circ}_{\sigma}$ is a boundary stratum of $\mathcal{M}_{n,k}$, together with ${\sigma_{k,n}}$. Based on poset data generated in \cite{Ferro:2020lgp}, we are emboldened to conjecture the following.

\begin{conjecture}\label{conj:mom-CW}
	We have a regular CW decomposition $$\mathcal{M}_{n,k} = \sqcup_{\sigma \in \mathcal{S}_{n,k}} \Phi^{\circ}_{\sigma}$$ of the momentum amplituhedron. In particular, each boundary stratum $\Phi^{\circ}_{\sigma}$ is homeomorphic to an open ball.
\end{conjecture}

The authors of \cite{Ferro:2020lgp} synthesised the boundary stratification of $\mathcal{M}_{n,k} $ for $4\le n\le8$ and $2\le k \le n-2$. From this data, they observed that positroid cells corresponding to momentum amplituhedron boundaries are labelled by Grassmannian forests. Combining their observations with knowledge about the singularity structure of tree-level amplitudes in $\mathcal{N}=4$ SYM, they hypothesised a bijection between boundaries and the aforementioned pictorial labels. This hypothesis is summarised below.

\begin{conjecture}[{\cite[Section 3.3]{Ferro:2020lgp}}]\label{conj:mom-cells}
	$\Phi_{G}^\circ$ is a boundary of $\mathcal{M}_{n,k}$ if and only if $G$ is a Grassmannian forest of type $(k,n)$. Moreover, given a Grassmannian forest $F$, $\dim \Phi_{F}^\circ = \dim_{\M}(F)$.
\end{conjecture}

\begin{remark}
	It is immediate from the definition that the momentum amplituhedron $\mathcal{M}_{n,2}$ is isomorphic to $\Gr^{\ge 0}_{2,n}$, and the boundary stratification	of $\mathcal{M}_{n,2}$ coincides with the boundary stratification of $\Gr^{\ge 0}_{2,n}$. 
\end{remark}

\subsection{Enumerating the boundaries of the momentum amplituhedron}

From \cref{conj:mom-cells}, the boundaries of the momentum amplituhedron $\mathcal{M}_{n,k}$ of dimension $r$ are in bijection with  the contracted Grassmannian forests of type $(k,n)$ with mom-dimension $r$, whose generating function was given in \cref{thm:GF-trees}. We summarise this result below.

\begin{corollary}\label{thm:GF-mom}
	The number of boundaries of the momentum amplituhedron $\mathcal{M}_{n,k}$ of dimension $r$ is given by $[x^ny^kq^r]\mathcal{G}_\text{forest}(x,y,q)$ where
	\begin{align*}
		x \mathcal{G}_\text{forest}(x,y,q)=\left(\frac{x}{1+\mathcal{G}_\text{tree}(x,y,q)}\right)^{\lrangle{-1}},
	\end{align*}
	was computed in \cref{thm:GF-trees} and the compositional inverse is with respect to the variable $x$. Equivalently, 
	\begin{align*}
		[x^n] \mathcal{G}_\text{forest}(x,y,q) = \frac{1}{n+1} [x^n] 
		\left( 1+\mathcal{G}_\text{tree}(x,y,q) \right)^{n+1}.
	\end{align*}
	In the above expressions, we have from 
	\cref{thm:GF-trees} that 
	\begin{align*}
		\mathcal{G}_\text{tree}(x,y,q)&=x\left(1+y+yq\,C^{\lrangle{-1}}(x,y,q)\right),\text{ with }\\
		C(x,y,q)&=\frac{x(1-x (1+y) q^2-x^2 y q^2 (1+q-q^2)-x^4 y^2 q^5 (1+q))}{(1+xq)(1+xyq)(1-xq^2)(1-xyq^2)}.
	\end{align*}
\end{corollary}

We have  checked that this formula reproduces all results for $[x^ny^k]\mathcal{G}_\text{forest}(x,y,q)$ listed in Tables 1 and 2 of \cite{Ferro:2020lgp}. Their results include as high as $(n,k)=(12,2)$. In \cref{tbl:G-forest} we present $[x^ny^k]\mathcal{G}_\text{forest}(x,y,q)$ for $n=4$ through $12$ and for all $2\le k\le \lfloor\frac{n}{2}\rfloor$.

The following corollaries are predicated upon \cref{conj:mom-cells}. \cref{thm:GF-mom-euler} additionally requires \cref{conj:mom-CW}.

\begin{corollary}
	The number of $0$-dimensional boundary strata of $\mathcal{M}_{n,k}$ is ${n \choose k}$, that is, $[x^n y^k q^0]\mathcal{G}_\text{forest}(x,y,q) = {n \choose k}$.
\end{corollary}

\begin{corollary}\label{thm:GF-mom-euler}
	The Euler characteristic of the momentum amplituhedron is $1$.
\end{corollary}

\begin{proof}[Proof of \cref{thm:GF-mom-euler}] 
	Recall that for a CW complex, the Euler characteristic is defined as the alternating sum $\chi=n_0-n_1+n_2-n_3+\ldots$ where $n_r$ denotes the number of cells of dimension $r$ in the complex. The Euler characteristic of $\mathcal{M}_{n,k}$ can be computed as $[x^ny^k]\mathcal{G}_\text{forest}(x,y,-1)$. To this end, let us specialize to $q=-1$ for 	$\mathcal{G}_\text{tree}(x,y,q)$ given in the expression \eqref{eq:GF-trees-1} from \cref{thm:GF-trees}. One can easily verify that
	\begin{align}
		C(x,y,-1)=\frac{x}{(1-x)(1-xy)},
	\end{align}
	and that its compositional inverse with respect to the variable $x$ is given by
	\begin{align}
		C^{\lrangle{-1}}(x,y,-1)=\frac{1+x(1+y)-\sqrt{\Delta}}{2xy},
	\end{align}
	where $\Delta=(1+x(1+y))^2-4x^2y$. Therefore we have that
	\begin{align*}
		\mathcal{G}_\text{forest}(x,y,-1)&= 
		\frac{1}{x} \left(\frac{x}{1+\mathcal{G}_\text{tree}(x,y,-1)}\right)^{\lrangle{-1}} \\
		&= \frac{1}{x} \left(\frac{x}{1+x(1+y)-\frac{1}{2} (1+x(1+y)-\sqrt{\Delta})} \right)^{\lrangle{-1}}\\
		&= \frac{1}{x} \left(\frac{2x}{1+x(1+y)+\sqrt{\Delta}} \right)^{\lrangle{-1}}\\
		&= \frac{1}{x} \left(\frac{1+x(1+y)-\sqrt{\Delta}}{2xy} \right)^{\lrangle{-1}}\\
		&= \frac{1}{x} C(x,y,-1)= \frac{1}{(1-x)(1-xy)} = 
		\sum_{0\le n}x^n\sum_{0\le k\le n} y^k,
	\end{align*}
	with every coefficient equal to $1$.
\end{proof}

%%%%%%%%%%%%%%%%%%%%%%%

\appendix

\section{Data}\label{sec:data}

In \cref{tbl:G-forest} we give expressions for $[x^ny^k]\mathcal{G}_\text{forest}(x,y,q)$ for $4\le n \le 12$, which extends the range of data presented in Tables 1 and 2 of \cite{Ferro:2020lgp}. Since the number of contracted Grassmannian forests of type $(k,n)$ equals those of type $(n-k,n)$ (see \cref{remark:tree-map}) we restrict $k$ to $2\le k\le \lfloor\frac{n}{2}\rfloor$.  

\begin{longtable}{p{1cm}p{\textwidth - 2cm}}
	\caption{Expressions for $[x^ny^k]\mathcal{G}_\text{forest}(x,y,q)$ for $4\le n \le 12$ and $2\le k\le \lfloor\frac{n}{2}\rfloor$.}
	\label{tbl:G-forest}
	\endfirsthead
	\endhead
	\toprule
	$(n,k)$ & \multicolumn{1}{c}{$[x^ny^k]\mathcal{G}_\text{forest}(x,y,q)$}\\
	\midrule
	\midrule
	$(4,2)$ & $q^4+4 q^3+10 q^2+12 q+6$ \\
	\midrule
	$(5,2)$ & $q^6+5 q^5+15 q^4+30 q^3+40 q^2+30 q+10$ \\
	\midrule
	$(6,2)$ & $q^8+6 q^7+21 q^6+50 q^5+90 q^4+120 q^3+110 q^2+60 q+15$ \\
	$(6,3)$ & $q^8+15 q^7+54 q^6+114 q^5+180 q^4+215 q^3+180 q^2+90 q+20$ \\
	\midrule
	$(7,2)$ & $q^{10}+7 q^9+28 q^8+77 q^7+161 q^6+266 q^5+350 q^4+350 q^3+245 q^2+105 q+21$ \\
	$(7,3)$ & $q^{10}+21 q^9+119 q^8+350 q^7+665 q^6+938 q^5+1050 q^4+910 q^3+560 q^2+210 q+35$ \\
	\midrule
	$(8,2)$ & $q^{12}+8 q^{11}+36 q^{10}+112 q^9+266 q^8+504 q^7+784 q^6+1008 q^5+1050 q^4+840 q^3+476 q^2+168 q+28$ \\
	$(8,3)$ & $q^{12}+28 q^{11}+188 q^{10}+720 q^9+1820 q^8+3262 q^7+4424 q^6+4788 q^5+4200 q^4+2870 q^3+1400 q^2+420 q+56$ \\
	$(8,4)$ & $q^{12}+32 q^{11}+300 q^{10}+1280 q^9+3264 q^8+5696 q^7+7420 q^6+7672 q^5+6426 q^4+4200 q^3+1960 q^2+560 q+70$ \\
	\midrule
	$(9,2)$ & $q^{14}+9 q^{13}+45 q^{12}+156 q^{11}+414 q^{10}+882 q^9+1554 q^8+2304 q^7+2898 q^6+3066 q^5+2646 q^4+1764 q^3+840 q^2+252 q+36$ \\
	$(9,3)$ & $q^{14}+36 q^{13}+279 q^{12}+1227 q^{11}+3726 q^{10}+8370 q^9+14322 q^8+19152 q^7+20622 q^6+18270 q^5+13230 q^4+7476 q^3+3024 q^2+756 q+84$ \\
	$(9,4)$ & $q^{14}+45 q^{13}+540 q^{12}+3003 q^{11}+10089 q^{10}+23049 q^9+38298 q^8+48618 q^7+49140 q^6+40656 q^5+27468 q^4+14490 q^3+5460 q^2+1260 q+126$ \\
	\midrule
	$(10,2)$ & $q^{16}+10 q^{15}+55 q^{14}+210 q^{13}+615 q^{12}+1452 q^{11}+2850 q^{10}+4740 q^9+6765 q^8+8340 q^7+8862 q^6+7980 q^5+5880 q^4+3360 q^3+1380 q^2+360 q+45$ \\
	$(10,3)$ & $q^{16}+45 q^{15}+395 q^{14}+1955 q^{13}+6705 q^{12}+17412 q^{11}+35640 q^{10}+58440 q^9+77490 q^8+84120 q^7+75852 q^6+57120 q^5+35280 q^4+17010 q^3+5880 q^2+1260 q+120$ \\
	$(10,4)$ & $q^{16}+60 q^{15}+880 q^{14}+5780 q^{13}+23385 q^{12}+65990 q^{11}+137835 q^{10}+220662 q^9+277890 q^8+281940 q^7+235410 q^6+163380 q^5+92862 q^4+41160 q^3+13020 q^2+2520 q+210$ \\
	$(10,5)$ & $q^{16}+65 q^{15}+1045 q^{14}+7915 q^{13}+34740 q^{12}+101240 q^{11}+212285 q^{10}+336220 q^9+415890 q^8+412980 q^7+336840 q^6+228102 q^5+126420 q^4+54600 q^3+16800 q^2+3150 q+252$ \\
	\midrule
	$(11,2)$ & $q^{18}+11 q^{17}+66 q^{16}+275 q^{15}+880 q^{14}+2277 q^{13}+4917 q^{12}+9042 q^{11}+14355 q^{10}+19855 q^9+24057 q^8+25542 q^7+23562 q^6+18480 q^5+11880 q^4+5940 q^3+2145 q^2+495 q+55$ \\
	$(11,3)$ & $q^{18}+55 q^{17}+539 q^{16}+2959 q^{15}+11275 q^{14}+32692 q^{13}+75735 q^{12}+143913 q^{11}+226908 q^{10}+297990 q^9+326997 q^8+301620 q^7+235158 q^6+154308 q^5+83160 q^4+34980 q^3+10560 q^2+1980 q+165$ \\
	$(11,4)$ & $q^{18}+77 q^{17}+1342 q^{16}+10109 q^{15}+46849 q^{14}+153527 q^{13}+380402 q^{12}+738067 q^{11}+1143780 q^{10}+1435005 q^9+1475562 q^8+1259412 q^7+901362 q^6+540540 q^5+265650 q^4+101640 q^3+27720 q^2+4620 q+330$ \\
	$(11,5)$ & $q^{18}+88 q^{17}+1782 q^{16}+16522 q^{15}+88924 q^{14}+318197 q^{13}+820512 q^{12}+1602986 q^{11}+2450437 q^{10}+2996972 q^9+2984520 q^8+2457840 q^7+1693230 q^6+975744 q^5+460152 q^4+168630 q^3+43890 q^2+6930 q+462$ \\
	\midrule
	$(12,2)$ & $q^{20}+12 q^{19}+78 q^{18}+352 q^{17}+1221 q^{16}+3432 q^{15}+8074 q^{14}+16236 q^{13}+28314 q^{12}+43252 q^{11}+58278 q^{10}+69564 q^9+73656 q^8+68904 q^7+56232 q^6+39072 q^5+22275 q^4+9900 q^3+3190 q^2+660 q+66$ \\
	$(12,3)$ & $q^{20}+66 q^{19}+714 q^{18}+4300 q^{17}+17985 q^{16}+57420 q^{15}+147202 q^{14}+312378 q^{13}+558855 q^{12}+850575 q^{11}+1104048 q^{10}+1221682 q^9+1152888 q^8+929610 q^7+640200 q^6+372636 q^5+178200 q^4+66495 q^3+17820 q^2+2970 q+220$ \\
	$(12,4)$ & $q^{20}+96 q^{19}+1950 q^{18}+16532 q^{17}+85875 q^{16}+316728 q^{15}+892496 q^{14}+2000988 q^{13}+3651945 q^{12}+5494500 q^{11}+6864858 q^{10}+7162452 q^9+6278283 q^8+4653000 q^7+2924460 q^6+1546776 q^5+670230 q^4+225720 q^3+54120 q^2+7920 q+495$ \\
	$(12,5)$ & $q^{20}+114 q^{19}+2808 q^{18}+30694 q^{17}+192327 q^{16}+805833 q^{15}+2456976 q^{14}+5730414 q^{13}+10553631 q^{12}+15676199 q^{11}+19057434 q^{10}+19172604 q^9+16115616 q^8+11412324 q^7+6835752 q^6+3438204 q^5+1413720 q^4+450450 q^3+101640 q^2+13860 q+792$ \\
	$(12,6)$ & $q^{20}+120 q^{19}+3120 q^{18}+36312 q^{17}+245007 q^{16}+1084332 q^{15}+3412852 q^{14}+8072064 q^{13}+14897368 q^{12}+22010536 q^{11}+26499066 q^{10}+26335804 q^9+21837585 q^8+15242040 q^7+8992236 q^6+4451832 q^5+1800414 q^4+563640 q^3+124740 q^2+16632 q+924$ \\
	\bottomrule
\end{longtable}

\section{Algebraic Relations}
\label{sec:alg-rel}

The rank-generating functions for contracted Grassmannian trees and forests presented in \cref{thm:GF-trees} satisfy the following algebraic relations:
\begin{align}
	\label{eq:GF-tree-rel}
	\begin{autobreak}
		0 
		= 
		q (q
		+1) \mathcal{G}_\text{tree}^5 
		+q x (q^2 x y
		-5 q (y
		+1)
		-5 (y
		+1)) \mathcal{G}_\text{tree}^4
		+ x^2 (q (-3 q^2 x y (y
		+1)
		+q (10
		-y ((x
		-10) y
		+x
		-19))
		+y (10 y
		+21)
		+10)
		+y)\mathcal{G}_\text{tree}^3
		+ x^3 (q^3 x y (y (3 y
		+7)
		+3)
		+q^2 (y
		+1) (y (2 (x
		-5) y
		+2 x
		-17)
		-10)
		-q (y (y (-x
		+10 y
		+32)
		+32)
		+10)
		-3 y (y
		+1))\mathcal{G}_\text{tree}^2
		- x^4 (q^3 x y (y
		+1) (y (y
		+4)
		+1)
		+q^2 (y
		+1) (y ((x
		-5) y^2
		+3 (x
		-4) y
		+x
		-12)
		-5)
		-q (y
		+1) (y (y (-x
		+5 y
		+16)
		+16)
		+5)
		-y (y (3 y
		+5)
		+3)) \mathcal{G}_\text{tree}
		+x^5 (q^3 x y^2 (y
		+1)^2
		-q^2 (y
		+1)^2 (y (y (-x
		+y
		+2)
		+2)
		+1)
		-q (y (y (y (-x
		+y (y
		+5)
		+10)
		+10)
		+5)
		+1)
		-y (y
		+1) (y^2
		+y
		+1))\,,
	\end{autobreak}\\
	\label{eq:GF-forest-rel}
	\begin{autobreak}
		0
		=
		q (x
		-1) x^2 y (x y
		-1) (q (x y
		+x
		-1)
		+x) (q (x y
		+x
		-1)
		+x y)\mathcal{G}_\text{forest}^6
		+ (q^3 x^2 y (x y
		+x
		-1) (x (x y (y
		+4)
		+x
		-5 (y
		+1))
		+4)
		+q^2 (x (x (-(x^3 (y
		+1) (y (y
		+2) (y^2
		+1)
		+1))
		+x^2 (y
		+1)^2 (y (5 y
		+3)
		+5)
		-2 x (y
		+1) (y (5 y
		+7)
		+5)
		+y (10 y
		+19)
		+10)
		-5 (y
		+1))
		+1)
		-q (x^2 (y^2
		+y
		+1)
		-2 x (y
		+1)
		+1) (x (x (x (y
		+1) (y (y
		+3)
		+1)
		-y (3 y
		+8)
		-3)
		+3 (y
		+1))
		-1)
		-x^2 y (x y
		+x
		-1) (x^2 (y^2
		+y
		+1)
		-2 x (y
		+1)
		+1))\mathcal{G}_\text{forest}^5
		+ (q^3 x^2 y (x^2 (y (3 y
		+7)
		+3)
		-9 x (y
		+1)
		+6)
		-q^2 (x (x (5 x^2 (y
		+1)^2 (y^2
		+y
		+1)
		-x (y
		+1) (y (20 y
		+31)
		+20)
		+30 y^2
		+57 y
		+30)
		-20 (y
		+1))
		+5)
		-q (x (x (x^2 (y (y (y (5 y
		+21)
		+31)
		+21)
		+5)
		-4 x (y
		+1) (y (5 y
		+11)
		+5)
		+30 y^2
		+63 y
		+30)
		-20 (y
		+1))
		+5)
		-x^2 y (x^2 (y (3 y
		+5)
		+3)
		-6 x (y
		+1)
		+3))\mathcal{G}_\text{forest}^4
		+(-10 q (q
		+1) x^3 y^3
		+(q
		+1) x^2 y^2 ((q (3 q
		-29)
		-3) x
		+30 q)
		+x y ((q
		+1) (q (3 q
		-29)
		-3) x^2
		+(q (q (57
		-4 q)
		+63)
		+3) x
		-30 q (q
		+1))
		-10 q (q
		+1) (x
		-1)^3)\mathcal{G}_\text{forest}^3
		+ (q
		+1) (q^2 x^2 y
		-10 q (x y
		+x
		-1)^2
		-x^2 y)\mathcal{G}_\text{forest}^2
		-5 q (q
		+1) (x y
		+x
		-1)\mathcal{G}_\text{forest}
		-q (q
		+1)\,.
	\end{autobreak}
\end{align}

%%%%%%%%%%%%%%%%%%%%%%%%%%%%%%%%%%%%%%%%%%%%%%%%%%%
%%%%%%%%%%%%%%%%%%%%%%%%%%%%%%%%%%%%%%%%%%%%%%%%%%%

% ACKNOWLEDGEMENTS
% Include acknowledgements to colleagues and referee here.
% Funding and grant support should appear in footnotes on the front page, using the 
% thanks command in the authors command (see above).

\section*{Acknowledgements}

The authors are grateful to  Francois Bergeron, Mireille Bousquet-M\'elou, 
Mathilde Bouvel, Colin Defant, Ira Gessel, Tomasz {\L}ukowski,  Alejandro Morales, 
Richard Stanley, and Einar Steingrimsson for providing useful comments and references. L.W.\ would like to acknowledge the support of the National Science Foundation under agreements No.\ DMS-1854316 and No.\ DMS-1854512. Any opinions, findings and conclusions or recommendations expressed in this material are those of the authors and do not necessarily reflect the views of the National Science Foundation.

%%%%%%%%%%%%%%%%%%%%%%%%%%%%%%%%%%%%%%%%%%%%%%%%%%%
%%%%%%%%%%%%%%%%%%%%%%%%%%%%%%%%%%%%%%%%%%%%%%%%%%%

% BIBLIOGRAPHY
% Pease provide us with a bibtex file for your bibliography
% You can find examples in ct-sample.bib. 
% Please use the correct entrytype 
%     article: any article published in a periodical like a journal article or magazine article
%     book: a book
%     booklet: like a book but without a designated publisher
%     conference: a conference paper
%     inbook: a section or chapter in a book
%     incollection: an article in a collection
%     inproceedings: a conference paper (same as the conference entry type)
%     manual: a technical manual
%     masterthesis: a Masters thesis
%     misc: used if nothing else fits
%     phdthesis: a PhD thesis
%     proceedings: the whole conference proceedings
%     techreport: a technical report, government report or white paper
%     unpublished: a work that has not yet been officially published
%
% When it exists, please add a DOI to your reference using the doi field, this will be shown as a link in the final pdf
% Similarly for preprints, please put the arXiv reference into the eprint field 
% When necessary, you can provide a link through the URL field. Please do not give both the URL and the DOI.
%
% We encourage you to use MathSciNet https://mathscinet.ams.org/mathscinet
% or other equivalent bibliography databases to get full and correct bibliographic entries.

\bibliographystyle{alphaurl}
\bibliography{bibliography}

\end{document}